\documentclass[11pt]{amsart}

\usepackage{amsmath}
\usepackage{amssymb}
\usepackage{amscd}
\usepackage{epsfig}

 \newcommand{\bee}{\begin{equation}}
 \newcommand{\eee}{\end{equation}}

\def\ov{\overline}

\def\b0{{\bf 0}}
\def\Span{{\rm Span}}
\def\Spec{\mbox{\em Spec}}

\newcommand{\be}{\begin{eqnarray}}
\newcommand{\ee}{\end{eqnarray}}
\newcommand{\supp}{\mbox{\rm supp}}

\newcommand{\eps}{{\mbox{$\epsilon$}}}

\newcommand{\Th}{{\theta}}
\newcommand{\gam}{{\gamma}}
\newcommand{\sig}{{\sigma}}

\newcommand{\R}{{\mathbb R}}
\newcommand{\Q}{{\mathbb Q}}
\newcommand{\Z}{{\mathbb Z}}
\newcommand{\C}{{\mathbb C}}

\newcommand{\Compl}{{\mathbb C}}

\newcommand{\Ak}{{\mathcal A}}

\newcommand{\Dk}{{\mathcal D}}

\newcommand{\Ck}{{\mathcal C}}

\newcommand{\Fk}{{\mathcal F}}

\newcommand{\Kk}{{\mathcal K}}

\newcommand{\Sk}{{\mathcal S}}
\newcommand{\Tk}{{\mathcal T}}

\newcommand{\Xt}{X_{\Tk}}

\newcommand{\bx}{{\bf x}}
\newcommand{\by}{{\bf y}}
\newcommand{\bz}{{\bf z}}
\newcommand{\bu}{{\bf u}}
\newcommand{\bv}{{\bf v}}

\newcommand{\bo}{{\bf 0}}

\newcommand{\dist}{\mbox{\rm dist}}

\newcommand{\Lam}{{\Lambda}}

\newcommand{\lam}{\lambda}

\newcommand{\om}{\omega}
\newcommand{\mmin}{\rm min}
\newcommand{\mmax}{\rm max}

\newcommand{\balpha}{\mbox{\boldmath{$\alpha$}}}
\newcommand{\bbeta}{\mbox{\boldmath{$\beta$}}}
\newcommand{\sbbeta}{\mbox{\tiny\boldmath${\beta}$}}
\newcommand{\bgamma}{\mbox{\boldmath{$\gamma$}}}

\def\Span{{\rm Span}}
\def\re{{\rm Re}}

\def\Diag{{\rm Diag}}

\def\ba{{\bf a}}

 \newtheorem{theorem}{Theorem}[section]
 \newtheorem{lemma}[theorem]{Lemma}
 \newtheorem{prop}[theorem]{Proposition}
 \newtheorem{cor}[theorem]{Corollary}

 \newtheorem{defi}[theorem]{Definition}
 \newtheorem{example}[theorem]{Example}
 \newtheorem{remark}[theorem]{Remark}

\numberwithin{equation}{section}

\begin{document}

\title{Pisot family self-affine tilings, discrete spectrum, and the Meyer property}

\author{Jeong-Yup Lee}
\address{Jeong-Yup Lee, KIAS, 87 Hoegiro, Dongdaemun-gu, Seoul 130-722, Korea}
\email{jylee@kias.re.kr}

\author{Boris Solomyak}
\address{Boris Solomyak, Box 354350, Department of Mathematics,
University of Washington, Seattle WA 98195}
\email{solomyak@math.washington.edu}

\begin{abstract}
We consider self-affine tilings in the Euclidean space and the
associated tiling dynamical systems, namely, the translation
action on the orbit closure of the given tiling. We investigate
the spectral properties of the system. It turns out that the
presence of the discrete component depends on the algebraic
properties of the eigenvalues of the expansion matrix $\phi$ for
the tiling. Assuming that $\phi$ is diagonalizable over $\C$ and
all its eigenvalues are algebraic conjugates of the same
multiplicity, we show that the dynamical system has a relatively
dense discrete spectrum if and only if it is not weakly mixing,
and if and only if the spectrum of $\phi$ is a ``Pisot family.''
Moreover, this is equivalent to the Meyer property of the
associated discrete set of ``control points'' for the tiling.
\end{abstract}

\date{\today}

\thanks{2000 {\em Mathematics Subject Classification:} Primary
37B50; Secondary 52C23
\\ \indent
{{\em Key words and phrases}: Pisot family, no weakly mixing,
self-affine tilings, discrete spectrum, Meyer sets}  \\ \indent
The research of J. L. was supported in part by KIAS and the
research of B. S. was supported in part by NSF}

\maketitle

\section{Introduction}

\noindent Given a self-affine tiling $\Tk$ of $\R^d$, we consider
the tiling space, or ``hull'' $X_{\Tk}$, defined as the orbit
closure of $\Tk$ in the ``local'' topology (please see the next
section for precise definitions and statements). The translation
action by $\R^d$ is uniquely ergodic, so we get a
measure-preserving tiling dynamical system $(X_{\Tk},\R^d,\mu)$.
We are interested in its spectral properties, specifically, in the
discrete component of the spectrum which may be defined as the
closed linear span of the eigenfunctions in $L^2(X_{\Tk},\mu)$. In
particular, we would like to know when the tiling system is {\em
weakly mixing}, which means absence of non-trivial eigenfunctions.

Our results give a complete answer to these questions in terms of
the {\em expansion matrix} $\phi$ of the tiling, under the
assumption that it is diagonalizable over $\C$ and its
eigenvalues are algebraic
conjugates of the same multiplicity. Let
$\Lam=\{\lam_1,\ldots,\lam_{d}\}=\Spec(\phi)$ be the set of (real and complex)
eigenvalues of $\phi$. It is known \cite{Ken.thesis,LS} that all
$\lam_i$ are algebraic integers. Following Mauduit \cite{Maud}, we
say that they form a {\em Pisot family} if for every $\lam\in
\Lam$ and every Galois conjugate $\lam'$ of $\lam$, if
$\lam'\not\in \Lam$, then $|\lam'| < 1$. We prove that
$(X_{\Tk},\R^d,\mu)$ has a relatively dense set of eigenvalues
(equivalently, the set of eigenvalues of full rank $d$) if and
only if $\Lam$ is a Pisot family, and this is also equivalent to
$(X_{\Tk},\R^d,\mu)$ being not weakly mixing.
An example shows that if the multiplicities of the eigenvalues of $\phi$ are not equal, even if
$\Spec(\phi)$ is a Pisot family, the set of eigenvalues of the tiling dynamical system
(not to be confused with $\Spec(\phi)$) may fail to be relatively dense in $\R^d$.

Special cases of our theorem were established earlier: for
self-similar tilings of $\R^d$, with $d\le 2$, in \cite{soltil},
and for self-similar tilings of $\R^d$ with a pure dilation
expanding matrix $\theta I$, in \cite{sol-eigen}. The present
paper covers a much more general self-affine case.

Additional motivation for our work comes from the theory of {\em
aperiodic order} and mathematics of quasicrystals. Considering
specially chosen ``control points'' in the tiles, we obtain a
Delone set $\Ck$, that is, a uniformly discrete and relatively
dense subset of $\R^d$, which is a {\em substitution Delone set},
see \cite{lawa}. (We should note that in geometric analysis Delone
sets are usually called {\em separated nets}.) It can be viewed as
an atomic configuration, and it turns out that its {\em
diffraction spectrum} is, in a certain precise sense, a ``part''
of the dynamical spectrum of the system $(X_{\Tk},\R^d,\mu)$, with
the Bragg peaks (sharp bright spots on the diffraction picture)
coming from the eigenvalues, see \cite{Dw,LMS,Gouere,BL}. Thus,
for instance, weak mixing implies that there are no Bragg peaks,
which indicates a certain level of ``disorder.''

A Delone set $Y\subset \R^d$ is {\em Meyer} if it is
relatively dense and $Y-Y$ is uniformly discrete. Answering a
question of Lagarias \cite{Lag}, we showed in \cite{LS} that
having a relatively dense set of Bragg peaks is equivalent to $Y$
being Meyer, for a primitive substitution Delone set associated
with a self-affine tiling. Our results in this paper imply that
this is also characterized by the Pisot family condition (under
our assumptions on $\phi$). The notion of {\em Meyer set} proved to be important
in the study of aperiodic order, see e.g. \cite{AL, Moody, LMS2,
Lee, BLM}. In \cite{AL}, under the Meyer set assumption for a
substitution tiling, a computational algorithm is developed to
decide whether the dynamical system has pure discrete spectrum.
It should be noted that we do not address the question of
pure discrete spectrum in the present paper but focus on the
discrete spectral component.

\subsection{Structure of the proof.}
A criterion for $\bx\in \R^d$ to be an eigenvalue of the system
$(X_{\Tk},\R^d,\mu)$ was obtained in \cite{soltil,sol-eigen}. From
it, the necessity of the Pisot family condition follows rather
easily, see \cite{soltil,Robi}. For the converse, we need
information on the location of control points, which is a
manifestation of certain ``rigidity'' of self-affine tilings. For
a self-similar tiling of the plane $\C\approx \R^2$ with a complex
expansion constant $\lam$, a result of Kenyon \cite{Ken} says that
the control points are contained in $\Z[\lam]\ba$, for some
$\ba\in \C$. We need an extension of this statement to the
higher-dimensional self-affine case, which is a key  result for us
(see Theorem~\ref{isomorphicImageOfC-contained} below). This
result does not depend on the Pisot family condition. We prove it
using the techniques developed by Thurston \cite{Thur} in the
2-dimensional self-similar case and extended to the
$d$-dimensional self-affine case in \cite{Ken.thesis,KS}, where a
necessary condition (which may be called the ``Perron family
condition'') for $\phi$ to be an expansion map was obtained. It is
here that we use the assumption that $\phi$ is diagonalizable: the
analog of the main theorem in \cite{KS} is open even for a
$2\times 2$ Jordan block.


\section{Definitions and statement of results}

\noindent We briefly review the basic definitions of tilings and
substitution tilings (see \cite{LMS2,Robi} for more details). We begin with a set of types (or colors)
$\{1,\ldots, \kappa\}$, which we fix once and for all. A {\em
tile} in $\R^d$ is defined as a pair $T=(A,i)$ where $A=\supp(T)$
(the support of $T$) is a compact set in $\R^d$ which is the
closure of its interior, and $i=l(T)\in \{1,\ldots, \kappa\}$ is
the type of $T$. We let $g+T = (g+A,i)$ for $g\in \R^d$. We say
that a set $P$ of tiles is a {\em patch} if the number of tiles in
$P$ is finite and the tiles of $P$ have mutually disjoint
interiors. A tiling of $\R^d$ is a set $\Tk$ of tiles such that
$\R^d = \cup \{\supp(T) : T \in \Tk\}$ and distinct tiles have
disjoint interiors. Given a tiling $\Tk$, finite sets of tiles of
$\Tk$ are called $\Tk$-patches. For $A \subset \R^d$, let
$[A]^{\Tk} = \{T \in \Tk : \supp(T) \cap A \neq \emptyset\}$.

We always assume that any two $\Tk$-tiles with the same color are
translationally equivalent. (Hence there are finitely many
$\Tk$-tiles up to translation.)

We say that a tiling $\Tk$ has {\em finite local complexity (FLC)}
if for each radius $R > 0$ there are only finitely many
translational classes of patches whose support lies in some ball
of radius $R$.

A tiling $\Tk$ is said to be {\em repetitive} if translations of any given patch occur uniformly dense in
$\R^d$; more precisely, for any $\Tk$-patch $P$, there exists $R>0$ such that every ball of radius $R$ contains
a translated copy of $P$.

Given a tiling $\Tk$, we define the {\em tiling space} as the
orbit closure of $\Tk$ under the translation action: $X_{\Tk} =
\ov{\{-g + \Tk:\,g\in \R^d\}}$, in the well-known ``local
topology'': for a small $\eps>0$ two tilings $\Sk_1,\Sk_2$ are
$\eps$-close if $\Sk_1$ and $\Sk_2$ agree on the ball of radius
$\eps^{-1}$ around the origin, after a translation of size less
than $\eps$. It is known that $X_\Tk$ is compact whenever $\Tk$
has FLC. Thus we get a topological dynamical system
$(X_{\Tk},\R^d)$ where $\R^d$ acts by translations. This system is
minimal (i.e.\ every orbit is dense) whenever $\Tk$ is repetitive.
Let $\mu$ be an invariant Borel probability measure for the
action; then we get a measure-preserving system $(X_{\Tk}, \R^d,
\mu)$. Such a measure always exists; under the natural assumption
of {\em uniform patch frequencies}, it is unique, see \cite{LMS}.
Tiling dynamical system have been investigated in a large number of papers; we do not provide an exhaustive bibliography, but mention a few: \cite{Petersen,CS,HRS,Kel}.
They have also been studied as translation surfaces or $\R^d$-solenoids \cite{BG,Gamb}.

\begin{defi}
{\em A vector $\balpha =(\alpha_1,\ldots,\alpha_d) \in \R^d$ is
said to be an {\em eigenvalue} for the $\R^d$-action if there
exists an eigenfunction $f\in L^2(X_{\Tk},\mu)$, that is, $\
f\not\equiv 0$ and for all $g\in \R^d$ and $\mu$-almost all
$\Sk\in X_{\Tk}$,
\begin{equation} \label{def-eig1}
f(\Sk-g) = e^{2 \pi i \langle g, \alpha\rangle} f(\Sk).
\end{equation}
Here $\langle \cdot,\cdot \rangle$ denotes the standard scalar
product in $\R^d$.}
\end{defi}

Note that this ``eigenvalue'' is actually a vector. In physics it might be
called a ``wave vector.'' We can also speak about eigenvalues for
the topological dynamical system $(X_{\Tk},\R^d)$; then the eigenfunction should be in
$C(X_{\Tk})$ and the equation (\ref{def-eig1}) should hold everywhere.

Next we define substitution tilings. Let $\phi$ be an expanding
linear map in $\R^d$, which means that all its eigenvalues are
greater than one in modulus.
The following definition is
essentially due to Thurston \cite{Thur}.

\begin{defi}\label{def-subst}
{\em Let $\Ak = \{T_1,\ldots,T_{\kappa}\}$ be a finite set of
tiles in $\R^d$ such that $T_i=(A_i,i)$; we will call them {\em
prototiles}. Denote by $\Ak^+$ the set of patches made of tiles
each of which is a translate of one of $T_i$'s. We say that
$\omega: \Ak \to \Ak^+$ is a {\em tile-substitution} (or simply
{\em substitution}) with expansion map $\phi$ if there exist
finite sets $\Dk_{ij}\subset \R^d$ for $i,j \le \kappa$, such that
\begin{equation}\label{subdiv}
\om(T_j)= \{u+T_i:\ u\in \Dk_{ij},\ i=1,\ldots, \kappa\} \ \ \
\mbox{for} \  j\le \kappa,
\end{equation}
with
$$
\phi A_j  = \bigcup_{i=1}^{\kappa} (\Dk_{ij}+A_i).
$$
Here all sets in the right-hand side must have disjoint interiors;
it is possible for some of the $\Dk_{ij}$ to be empty.}
\end{defi}

The substitution (\ref{subdiv}) is extended to all translates of
prototiles by $\om(x+T_j)= \phi x + \om(T_j)$, and to patches and
tilings by $\om(P)=\cup\{\om(T):\ T\in P\}$. The substitution
$\om$ can be iterated, producing larger and larger patches
$\om^k(T_j)$. To the substitution $\om$ we associate its $\kappa
\times \kappa$ substitution matrix with the entries $\sharp
(\Dk_{ij})$. The substitution $\om$ is called {\em primitive} if
the substitution matrix is primitive. We say that $\Tk$ is a fixed
point of a substitution if $\om(\Tk) = \Tk$.

\begin{defi} \label{def-saf}
{\em A repetitive fixed point of a primitive tile-substitution
with FLC is called a {\em self-affine tiling}. It is called {\em
self-similar} if the expansion map is a similitude, that is,
$|\phi(x)| = \Th |x|$ for all $x\in \R^d$, with some $\Th>1$.}
\end{defi}

\begin{remark}
{\em 1. A fixed point of a  primitive tile-substitution is not
necessarily of finite local complexity, see \cite{Danzer,FraRob}.
Thus we have to assume FLC explicitly.

2. It is well-known (and easy to see, e.g.\ in the one-dimensional
case) that the fixed point may be repetitive even for a
non-primitive substitution. Conversely, the fixed point of a
primitive substitution need not be repetitive. (However, if the
tile-substitution is primitive and the fixed point tiling $\Tk$
has a tile which contains the origin in its interior, then $\Tk$
is repetitive \cite{Prag}.)

3. For a self-similar tiling of $\R^d$, with $d\le 2$, we can
speak of an expansion factor; it is a real number if $d=1$ and a
complex number if $d=2$ (we then view the plane as a complex
plane).}

\end{remark}

An important question, first raised by Thurston \cite{Thur}, is to
characterize which expanding linear maps may occur as expansion
maps for self-affine (self-similar) tilings. It is pointed out in
\cite{Thur} that in one dimension, $\lam$ is an expansion factor
if and only if $\Th=|\lam|$ is a {\em Perron number}, that is, an
algebraic integer greater than one whose Galois conjugates are all
strictly less than $\Th$ in modulus (necessity follows from the
Perron-Frobenius theorem and sufficiency follows from a result of
Lind \cite{Lind}). In two dimensions, Thurston \cite{Thur} proved
that if $\lam$ is a  complex expansion factor of a self-similar
tiling, then $\lam$ is a {\em complex Perron number}, that is, an
algebraic integer whose Galois conjugates, other than $\ov{\lam}$,
are all less than $|\lam|$ in modulus. The following theorem was
stated in \cite{Ken.thesis}, but complete proof was not available
until recently.

\begin{theorem} \cite{Ken.thesis,KS} \label{th-KS}
Let $\phi$ be a diagonalizable (over $\Compl$) expansion map on
$\R^d$, and let $\Tk$ be a self-affine tiling of $\R^d$ with
expansion $\phi$. Then

{\bf (i)} every eigenvalue of $\phi$ is an algebraic integer;

{\bf (ii)}
 if $\lam$ is an eigenvalue of $\phi$ of multiplicity $k$ and $\gam$ is
an algebraic conjugate of $\lam$, then either $|\gam| < |\lam|$, or $\gam$
is also an eigenvalue of $\phi$ of multiplicity greater or equal to $k$.
\end{theorem}

\begin{remark}
{\em 1. Note that if $|\gam|=|\lam|$ in part (ii) of the theorem,
then the multiplicities of $\gam$ and $\lam$ are the same.

2. It is conjectured that the condition on $\phi$ in the theorem
is also sufficient. There are partial results in this direction
\cite{Kenyon.construction}; see \cite{KS} for a discussion.}
\end{remark}

For a self-affine tiling $\Tk$, the
corresponding tiling dynamical system $(X_{\Tk}, \R^d)$ is
uniquely ergodic, see \cite{LMS2, Robi}. Denote by $\mu$ the
unique invariant probability measure.
There is a rich structure associated with self-affine tiling dynamical systems.
As a side remark, we mention that the substitution map $\om$ extends to an endomorphism of the tiling space,
which is hyperbolic in a certain sense, see \cite{AP}.
The partition of the tiling space according to the type of the tile containing the origin provides a Markov partition for $\om$. The situation is especially nice
when $\Tk$ is non-periodic, which is equivalent to $\om$ being invertible \cite{solucp}. In order to state our results we need the following.

\begin{defi} \cite{Maud}
{\em A set $\Lam$ of algebraic integers is called a {\em Pisot
family} if for every $\lam\in \Lam$, if $\gam$ is an algebraic
conjugate of $\lam$ and $\gam\not\in \Lam$, then $|\gam|< 1$.
Otherwise $\Lam$ is called {\em non-Pisot}.}
\end{defi}

In this paper we assume that:
\begin{itemize}
\item all the eigenvalues of $\phi$ are algebraic conjugates
    with the same multiplicity.
\end{itemize}
Let $Spec(\phi)$ be the set of all eigenvalues of $\phi$ (the spectrum of $\phi$).
By assumption, there exists a monic irreducible polynomial
$p(t) \in \Z[t]$ (the minimal polynomial) such that $p(\lam) = 0$
for all $\lam \in Spec(\phi)$.

\begin{theorem} \label{th-main}
Let $\Tk$ be a self-affine tiling of $\R^d$ with a diagonalizable
expansion map $\phi$. Suppose that all the eigenvalues of $\phi$
are algebraic conjugates with the same multiplicity. Then
the following are equivalent:\\
{\bf (i)} The set of eigenvalues of $(X_{\Tk},\R^d,\mu)$ is
    relatively dense  in $\R^d$.\\
{\bf (ii)} $Spec(\phi)$ is a Pisot family. \\
{\bf (iii)} The system $(X_{\Tk},\R^d,\mu)$ is not weakly
    mixing (i.e., it has eigenvalues other than ${\bf 0}$).
\end{theorem}

\begin{remark}
{\em 1. In part (i) we could equally well talk about the
topological dynamical system $(X_{\Tk},\R^d)$ since every
eigenfunction may be chosen to be continuous \cite{sol-eigen}.

2. The necessity of the Pisot family condition for self-affine tiling systems that are not weakly mixing
was proved by Robinson \cite{Robi} in a more
general case; it is a consequence of \cite{soltil}.}
\end{remark}

\begin{example} {\em (i) In Fig.\,2 and Fig.\,3 of \cite{KS} a self-affine tiling $\Tk_1$ is given, with the diagonal expansion matrix $\Diag[\lam_1,\lam_2]$ where
$\lam_1\approx 2.19869$ and $\lam_2\approx -1.91223$ are roots of the polynomial $x^3-x^2-4x+3$. Observe that $\{\lam_1,\lam_2\}$ is a Pisot family, hence
the set of eigenvalues for the associated dynamical system is relatively dense in $\R^2$.

(ii) The assumption of equal multiplicity cannot be dropped from Theorem~\ref{th-main}.
Indeed, consider the tiling $\Tk$ which is a ``direct product'' of $\Tk_1$ defined in (i) and
a self-similar tiling $\Tk_2$ of $\R$ with expansion $\lam_1$.
Such a tiling $\Tk_2$ exists by \cite{Lind} (see \cite{Solnotes} for more details)
since $\lam_1$ is a Perron number. Direct product substitution tilings have been studied by S. Mozes \cite{mozes} and N. P. Frank \cite{natalie}. It is
easy to see that the set of eigenvalues for the dynamical system $(\Xt,\R^3)$ is obtained as a direct sum of those which
correspond to the systems $(X_{\Tk_1},\R^2)$ and $(X_{\Tk_2},\R)$. By \cite{soltil}, the system $(X_{\Tk_2},\R)$ is
weakly mixing, because $\lam_1$ is not a Pisot number. Thus, the tiling $\Tk$ has expansion map
$\phi=\Diag[\lam_1,\lam_2,\lam_1]$ for which $\Spec(\phi)$ is a Pisot family, but the associated dynamical
system does not have a relatively dense set of eigenvalues. }
\end{example}

Next we state our result on Meyer sets.
Recall that a Delone set is a relatively dense and uniformly
discrete subset of $\R^d$.

\begin{defi}
{\em A Delone set $Y$ is called a {\em Meyer set} if
$Y-Y$ is uniformly discrete.}
\end{defi}

There is a standard way to choose distinguished points in the tiles of a self-affine tiling so that they form a $\phi$-invariant Delone set. They are called
{\em control points}.

\begin{defi} \cite{Thur,Prag}
{\em Let $\Tk$ be a fixed point of a primitive substitution with
expansion map $\phi$. For each $\Tk$-tile $T$, fix a tile $\gamma
T$ in the patch $\omega (T)$; choose $\gamma T$ with the same
relative position for all tiles of the same type. This defines a
map $\gamma : \Tk \to \Tk$ called the {\em tile map}. Then define
the {\em control point} for a tile $T \in \Tk$ by
\[ \{c(T)\} = \bigcap_{n=0}^{\infty} \phi^{-n}(\gamma^n T).\]
}
\end{defi}

\noindent The control points have the following properties:
\begin{itemize}
\item[(a)] $T' = T + c(T') - c(T)$, for any tiles $T, T'$ of the
same type; \item[(b)] $\phi(c(T)) = c(\gamma T)$, for $T \in \Tk$.
\end{itemize}
Control points are also fixed for tiles of any tiling $\mathcal{S}
\in X_{\Tk}$: they have the same relative position as in
$\Tk$-tiles. Note that the choice of control points is non-unique, but there are only finitely many possibilities, determined by the choice of the tile map.

Let
\[ \Ck:= \Ck(\Tk) = \{ c(T) : T \in \Tk \}\]
be a set of control points of the tiling $\Tk$ in $\R^d$. Let
$$
\Xi := \Xi(\Tk) = \bigcup_{i=1}^{\kappa} (\Ck_i - \Ck_i),
$$
where $\Ck_i$ is the set of control points of tiles of type $i$.
Equivalently, $\Xi$ is the set of translation vectors between two
$\Tk$-tiles of the same type.

\begin{cor} \label{cor-Meyer}
Let $\Tk$ be a self-affine tiling of $\R^d$ with a diagonalizable
expansion map $\phi$. Suppose that all the eigenvalues of $\phi$
are algebraic conjugates with the same multiplicity. Then the set
of control points $\Ck$ is Meyer if and only if $Spec(\phi)$ is a
Pisot family.
\end{cor}

This is an immediate consequence of Theorem~\ref{th-main} and
\cite[Th.\,4.14]{LS}.


\section{Preliminaries}
\noindent Recall that $\phi$ is assumed to be diagonalizable over
$\C$. For a complex eigenvalue $\lam$ of $\phi$, the $2 \times 2$
diagonal block $\left[
\begin{array}{ll}
         \lambda & 0 \\
         0 & \overline{\lambda}
         \end{array} \right] $ is similar to a real $2 \times 2$ matrix
\be \label{eq-mult}
          \left[ \begin{array}{rr}
         a & -b \\
         b & a
         \end{array} \right] = S^{-1}\left[\begin{array}{ll}
         \lambda & 0 \\
         0 & \overline{\lambda}
         \end{array} \right]S,
\ee where $\lam = a + ib$, $a, b \in
         \R$, and $S = \frac{1}{\sqrt{2}}\left[\begin{array}{rr} 1 & i \\ 1 & -i \end{array} \right]$.
Since $\phi$ is diagonalizable
over $\C$, we can assume, by appropriate choice of basis, that
$\phi$ is in the real canonical form of the linear map, see
\cite[Th.\,6.4.2]{HS}. This means that $\phi$ is block-diagonal,
with the diagonal entries equal to $\lam$ corresponding to real
eigenvalues, and diagonal $2 \times 2$ blocks of the form $\left[
\begin{array}{rr}
         a_j & -b_j \\
         b_j & a_j
         \end{array} \right]$ corresponding to complex eigenvalues
         $a_j + i b_j$.

Let $J$ be the multiplicity of each eigenvalue of $\phi$. We can
write
\[ \phi = \left[
\begin{array}{ccc}
                  \psi_1 & \cdots & {O} \\
                  \vdots & \ddots & \vdots \\
                  {O} & \cdots & \psi_J
                  \end{array} \right] \  \mbox{and} \ \
                  \psi_j = \psi := \left[
                  \begin{array}{ccc}
                  A_{1} & \cdots & 0 \\
                  \vdots & \ddots & \vdots \\
                  0 & \cdots & A_{s+t}
                  \end{array} \right] \ \ \ \mbox{for any $1 \le j \le J$} \]
where $A_{k}$ is a real $1\times 1$ matrix for $1 \le k \le s$, a
real $2\times 2$ matrix of the form $\left[
\begin{array}{lr}
         a_{k} & -b_{k} \\
         b_{k} & a_{k}
        \end{array} \right] $
for $s + 1 \le k \le s + t$, and ${O}$ is the $(s+2t) \times (s+2t)$
zero matrix. Then the eigenvalues of $\psi$ are
$$\lam_{1}, \ldots, \lam_{s},\lam_{s+1},\ov{\lam}_{s+1},\ldots,\lam_{s + t},\ov{\lam}_{s+t}.$$
Let $m: = s +2t$; this is the size of the matrix $\psi$.
For each $1 \le j \le J$, let
\begin{equation}
H_{j} = \{0\}^{(j-1)m} \times \R^{m} \times
\{0\}^{d- jm}\,. \nonumber
\end{equation}
Further, for each $H_{j}$ we have the direct sum decomposition
$$
H_j = \bigoplus_{k=1}^{s + t} E_{jk},
$$
such that each $E_{jk}$ is $\phi,\phi^{-1}$-invariant and
$\phi|_{E_{jk}} \approx A_{k}$, identifying $E_{jk}$ with $\R$ or
$\R^2$. Define a norm on $\R^d$ by
\begin{equation} \label{def-norm}
\|\bx\| = \max_{j,k}\|\bx_{jk}\|\ \ \mbox{for}\
\bx= \sum_{j=1}^J \sum_{k=1}^{s+t} \bx_{jk},\ \bx_{jk} \in E_{jk},
\end{equation}
where $\|\bx_{jk}\|$ is the Euclidean norm on $E_{jk}$, so that
$\|\phi \bx_{jk}\| = |\lam_{jk}| \|\bx_{jk}\|$. Let
$$
\Q[\phi]:= \{p(\phi):\ p \in \Q[x]\},\ \ \Z[\phi]:= \{p(\phi):\ p \in \Z[x]\}.
$$
Let $P_{j}$ be the canonical projection of $\R^d$ onto $H_{j}$
such that \be \label{def-projection} P_{j}(\bx) = \bx_{j}, \ee
where $\bx = \bx_{1} + \cdots + \bx_{J}$ and $\bx_{j} \in H_{j}$
with $1 \le j \le J$. Let $\phi_{j} = \phi|_{H_{j}}$.

\medskip

We define $\balpha_{j} \in H_{j}$ such that for each $1 \le n \le
d$, \be \label{def-alpha} (\balpha_{j})_{n} =
\left\{\begin{array}{ll}
                             1 \ \ \ & \mbox{if} \ \  (j-1)m + 1 \le n \le jm;\\
                             0 \ \ \ & \mbox{else} \,.
                           \end{array} \right.
\ee

The next theorem is a key result of the paper; it is the
manifestation of rigidity alluded in the Introduction.

\begin{theorem} \label{isomorphicImageOfC-contained}
Let $\Tk$ be a self-affine tiling of $\R^d$ with a diagonalizable
expansion map $\phi$. Suppose that all the eigenvalues of $\phi$
are algebraic conjugates with the same multiplicity. Then there
exists an isomorphism $\rho: \R^d \to \R^d$ such that
\[ \mbox{$\rho \phi =
\phi \rho$ \ \ \ and} \ \ \ \mathcal{C} \subset \rho(\Z[\phi] \balpha_{1} + \cdots +
\Z[\phi] \balpha_{J})\,,\] where $\balpha_j$, $1 \le j \le J$, are given as above.
\end{theorem}

The reason we call this ``rigidity'' is by analogy with \cite[Th.\,9]{Ken} (see the discussion at the beginning of the proof in \cite{Ken}).

We give a proof of Theorem~\ref{isomorphicImageOfC-contained} in Section \ref{Structure of
control point set} below and make use of it in proving the main
theorem in Section \ref{Meyer-property}. Note that the choice of
$\balpha_j$ is rather arbitrary; it is ``hidden'' in the linear
isomorphism $\rho$.

Now we continue with the preliminaries; we
need to handle the real and complex eigenvalues a little bit
differently.
Consider the linear injective map $\Fk:\R^m\,\to \R^s\oplus
\C^{2t}$ given by \be \label{def-Fk} \Fk(x_1,\ldots, x_s,
x_{s+1},\ldots,x_{s+2t}) =\ee$$ =\left(x_1,\ldots, x_s,
\frac{x_{s+1}+ix_{s+2}}{\sqrt{2}},
\frac{x_{s+1}-ix_{s+2}}{\sqrt{2}},\ldots, \frac{x_{s+2t-1} +
ix_{s+2t}}{\sqrt{2}},
\frac{x_{s+2t-1}-ix_{s+2t}}{\sqrt{2}}\right)\,. $$ In other words,
identifying $H_j$ with $\R^m$, we apply the transformation $S$
from (\ref{eq-mult}) in every subspace $E_{jk}$,
$k=s+1,\ldots,s+t$. In view of (\ref{eq-mult}), we have \be
\label{eq-mult2} \Fk(\psi \bx) = D \Fk(\bx)\ \ \ \mbox{and}\ \ \
\Fk(\psi^T \bx) = \ov{D} \Fk(\bx), \ee where \be \label{def-diag}
D=\Diag[\lam_1,\ldots,\lam_s,\lam_{s+1},\ov{\lam}_{s+1}\ldots,
\lam_{s+t},\ov{\lam}_{s+t}]\ee is a diagonal matrix.

\medskip

The following lemma is well-known and easy to prove using the
Vandermonde matrix.

\begin{lemma} \label{independent-phi-alphas-I}
Let $D$ be a diagonal matrix on $\C^m$ with distinct complex
eigenvalues. Let $\bz = [z_1, \cdots, z_m ]^T \in \C^m$ be such
that $z_k \neq 0$ for all  $1 \le k \le m$. Then $\{\bz, D \bz,
\cdots, D^{m-1} \bz \}$ is linearly independent over $\C$.
\end{lemma}

\begin{cor} \label{independent-phi-alphas}
Suppose that $\bx \in \R^m$ is such that $\bz=\Fk(\bx)$ has all
$(m)$ non-zero coordinates. Then both $\{\bx,
\psi\bx,\ldots,\psi^{m-1}\bx\}$ and $\{\bx, \psi^T \bx,\ldots,
(\psi^T)^{m-1}\bx\}$ are linearly independent over $\R$.
\end{cor}

\begin{sloppypar}
\begin{proof}
We have
$$
\Fk(\{\bx,\ldots,\psi^{m-1}\bx\}) = \{\bz,\ldots,D^{m-1}\bz\}
$$
by (\ref{eq-mult2}). By Lemma~\ref{independent-phi-alphas-I}, the
set $\{\bz,\ldots,D^{m-1}\bz\}$ is independent over $\C$ and hence
$\{\bx, \psi\bx,\ldots,\psi^{m-1}\bx\}$ is independent over $\R$,
using the fact that $\Fk$ is injective. The proof for the second set (with
transpose matrices) is exactly the same. \end{proof}
\end{sloppypar}

\begin{cor} \label{basisWithAlphas}
The set $W := \{\balpha_{1}, \ldots, \phi^{m-1} \balpha_{1},
\ldots, \balpha_{J}, \ldots, \phi^{m-1}\balpha_{J}\}$ forms a
basis of $\R^d$.
\end{cor}

\begin{proof}
Identifying $H_j$ with $\R^m$, we have $\phi_j =\phi|_{H_j}\approx
\psi$ and use the isomorphism $\Fk$ defined above.  In view of
(\ref{def-alpha}), all the components of
$\bz_j=\Fk(\balpha_j)$ are non-zero, so the claim follows from
Corollary~\ref{independent-phi-alphas}.
\end{proof}

For $\bx,\by\in \R^m$ we use the standard scalar product $\langle
\bx,\by \rangle = \sum_{k=1}^m x_k y_k$, and for $\bz,\bu \in
\R^{s} \oplus \C^{2t}$ the scalar product is given by
$$
\langle \bz,\bu \rangle_{_\C} = \sum_{k=1}^{s+2t} z_k \ov{u}_k.
$$
Observe that \be \label{eq-scalar} \langle \bx,\by \rangle =
\langle \Fk(\bx),\Fk(\by) \rangle_{_\C}\ \ \ \mbox{for all}\ \
\bx,\by\in \R^m. \ee Recall also that for any $m \times m$ matrix
$A$,
$$
\langle A\bx,\by \rangle = \langle \bx,A^T\by \rangle\ \ \ \mbox{for all}\ \ \bx,\by\in \R^m.
$$

\section{Proof of the main theorem (proof of Theorem\,\ref{th-main})} \label{Meyer-property}

\noindent Here we deduce Theorem\,\ref{th-main} from
Theorem\,\ref{isomorphicImageOfC-contained}. Recall that a set of
algebraic integers $\Theta = \{\theta_1, \cdots, \theta_r \}$ is a
{\em Pisot family} if for any $1 \le j \le r$, every Galois
conjugate $\gamma$ of $\theta_j$ with $|\gamma| \ge 1$ is
contained in $\Theta$. We denote by $\dist(x,\Z)$ the distance from a real number $x$ to the nearest integer.

\begin{prop} \label{set-of-eigenvalues}
Let $\Tk$ be a self-affine tiling of $\R^d$ with a diagonalizable
expansion map $\phi$. Suppose that all the eigenvalues of $\phi$
are algebraic conjugates with the same multiplicity. If
$Spec(\phi)$ is a Pisot family, then the set of eigenvalues of
$(X_{\Tk}, \R^d, \mu)$ is relatively dense.
\end{prop}

\begin{proof} Recall that $\Xi = \{\bx \in \R^d :\ \exists\ T \in \Tk,\ T+\bx
\in \Tk\}$ is the set of ``return vectors'' for the tiling $\Tk$,
and let $\mathcal{K} = \{\bx \in \R^d :\ \Tk - \bx = \Tk \}$ be
the set of translational periods. Clearly, $\Kk \subset \Xi
\subset \Ck-\Ck$.  We know from \cite[Th.\,3.13]{sol-eigen} that
$\gamma$ is an eigenvalue for $(X_{\Tk}, \R^d, \mu)$ if and only
if \[ \lim_{n \to \infty} e^{2 \pi i \langle \phi^n {\bf x},
\gamma \rangle} = 1 \ \ \ \mbox{for all ${\bf x} \in \Xi$ \;\;
and}\]
\[e^{2 \pi i \langle {\bf x},
\gamma \rangle} = 1 \ \ \ \mbox{for all ${\bf x} \in
\mathcal{K}$}.\]

Let $\balpha_{j}\in H_{j}$ be the vectors from (\ref{def-alpha}).
Consider them as vectors in $\R^m$, and let $\Fk$ be the linear
map $\R^m \to \R^s\oplus \C^{2t}$ given by (\ref{def-Fk}). Recall
that $\phi_j=\phi|_{H_j}$ has $s$ real and $2t$ complex
eigenvalues, and $m = s+2t$. Define $\bbeta_{j}\in H_{j}\approx
\R^m$ so that
$$
(\Fk(\bbeta_{j}))_k = (\ov{\Fk(\balpha_{j})})_k^{-1}\ \ \mbox{for}\ 1 \le k \le m.
$$
More explicitly,
$$
(\bbeta_j)_k = (\balpha_j)_k^{-1}\ \ \mbox{for}\ 1 \le k \le s
$$
and $$ (\bbeta_j)_{s+2k-1} \pm i(\bbeta_j)_{s+2k} =
\frac{2}{(\balpha_j)_{s+2k-1} \mp i(\balpha_j)_{s+2k}}\ \
\mbox{for}\ 1 \le k \le t.
$$
\smallskip
\noindent
Note that $\bbeta_{j}\in H_{j}$ are well-defined,  and
$\Fk(\bbeta_{j})$ have all non-zero coordinates  in $H_{j}$. Thus,
\be \label{prebasis} B_{j} := \{\bbeta_{j}, \ldots,
(\phi^T)^{m-1}\bbeta_{j}\} \ee is a basis of $H_{j}$ by
Corollary~\ref{independent-phi-alphas} (note that $H_j$ is also
$\phi^T$-invariant and $\phi^T|_{H_j}$ is isomorphic to $\psi^T$).
It follows that $B:=\bigcup_{j=1}^J B_j$ is a basis of $\R^d$. We
will show that all elements of the set $(\rho^T)^{-1}(\phi^T)^K B$ are
eigenvalues for the tiling dynamical system, for $K$ sufficiently
large.

By the definition of $\bbeta_{j}$, in view of (\ref{eq-scalar})
and (\ref{eq-mult2}), for any $n \in \Z_{\ge 0}$ and $0 \le l <
m$,
\begin{eqnarray}
\langle \phi^n \balpha_{j}, (\phi^T)^l \bbeta_{j} \rangle & = & \langle \phi^{n+l}\balpha_j, \bbeta_j\rangle \nonumber \\
                 & = & \langle \Fk(\phi^{n+l}\balpha_j), \Fk(\bbeta_i)\rangle_\C \nonumber \\
                 & = & \langle D^{n+l} \Fk(\balpha_i), \Fk(\bbeta_i)\rangle_{_\C} \nonumber \\
                 & = & \sum_{k=1}^s \lam_k^{n+l} + \sum_{k=s+1}^{s+t} (\lam_k^{n+l}+\ov{\lam}_k^{n+l}). \label{eq-PV}
\end{eqnarray}
Here $D$ is the diagonal matrix from (\ref{def-diag}).
Since
$Spec(\phi)$ is a Pisot family, it follows
that $\dist(\langle \phi^n \balpha_{j},
(\phi^T)^l \bbeta_{j} \rangle,\Z)\to 0$, as $n\to \infty$. (This is a standard argument: the sum of $(n+l)$-th powers of all
zeros of a polynomial in $\Z[x]$ is an integer, hence the distance from the sum in (\ref{eq-PV}) to $\Z$ is
bounded by the sum of the moduli of $(n+l)$-th powers of their remaining conjugates, which are all less than one in modulus.
Thus, this distance tends to zero exponentially fast.)
Observe also that $\langle
\phi^n \balpha_{u}, \bbeta_{j}\rangle = 0$ if $u\ne j$, hence
\[ \lim_{n \to \infty} e^{2 \pi i \langle \phi^n \by, (\phi^T)^l {\sbbeta}_{j} \rangle} = 1
\ \ \ \mbox{for all} \ \by \in \Z[\phi]\balpha_{1} + \cdots +
\Z[\phi]\balpha_{J} \,. \] Therefore, by
Theorem~\ref{isomorphicImageOfC-contained}, using that $\Xi
\subset \Ck - \Ck$, we obtain \be
\label{first-eigenvalue-condition} \lim_{n \to \infty} e^{2 \pi i
\langle \phi^n \bx, (\rho^T)^{-1}(\phi^T)^l {\sbbeta}_{j} \rangle}
= 1 \ \ \ \mbox{for all} \ \bx \in \Xi \,. \ee Furthermore, by
\cite[Cor.\,4.4]{sol-eigen}, the convergence is uniform in $\bx
\in \Xi$, that is,
\[\lim_{n \to \infty} \sup_{\bx \in \Xi} |e^{2\pi i \langle \phi^n \bx, (\rho^T)^{-1}(\phi^T)^l {\sbbeta}_{j} \rangle} - 1| = 0 \,.\]
Recall that $\mathcal{K} \subset \Xi$, and $\mathcal{K}$ is a
discrete subgroup in $\R^d$. So for every $\bx \in \Kk$, $$
\lim_{n \to \infty} \sup_{\stackrel{k \in \Z_+}{\bx \in
\mathcal{K}}} |e^{2\pi i \langle \phi^n  (k \bx),
(\rho^T)^{-1}(\phi^T)^l {\sbbeta}_{j} \rangle} - 1| = \lim_{n \to
\infty} \sup_{\stackrel{k \in \Z_+}{\bx \in \mathcal{K}}} |e^{2\pi
i k \langle \phi^n \bx, (\rho^T)^{-1}(\phi^T)^l {\sbbeta}_{j}
\rangle} - 1| = 0. $$ It follows that there exists $K_{l} \in
\Z_+$ such that for any $n \ge K_{l}$, for all $\bx\in \Kk$, \be
\label{supremumOfexp-eigenvalue} \sup_{\stackrel{k \in \Z_+}{\bx
\in \mathcal{K}}} |e^{2\pi i k \langle \phi^n \bx,
(\rho^T)^{-1}(\phi^T)^l {\sbbeta}_{j} \rangle} - 1| < 1/2 \,. \ee
However, unless $\langle \phi^n \bx, (\rho^T)^{-1}(\phi^T)^l
\bbeta_{j} \rangle \in \Z$ for all $\bx \in \mathcal{K}$,
(\ref{supremumOfexp-eigenvalue}) does not hold. Thus
\[e^{2 \pi i \langle \phi^n \bx, (\rho^T)^{-1}(\phi^T)^l \sbbeta_{j} \rangle} = e^{2 \pi i \langle \bx, (\rho^T)^{-1}(\phi^T)^{n+l} \sbbeta_{j} \rangle} = 1 \ \
\mbox{for all $\bx \in \mathcal{K}$ and all $n \ge K_{l}$}\,.\]
Let $ K = {\mmax} \{K_{l}:\ 0 \le l < m \}. $ Then \be
\label{second-eigenvalue-condition} e^{2 \pi i \langle \bx,
(\rho^T)^{-1}(\phi^T)^{K+l} \sbbeta_{j} \rangle} = 1 \ \ \mbox{for
all $\bx \in \mathcal{K}$}\,. \ee So from
(\ref{first-eigenvalue-condition}) and
(\ref{second-eigenvalue-condition}) it follows that
$(\rho^T)^{-1}(\phi^T)^{K+l} \bbeta_{j}$ is an eigenvalue of
$(X_{\Tk}, \R^d, \mu)$ for $l = 0,\ldots,m-1$. We have shown that
all vectors of the set $(\rho^T)^{-1}(\phi^T)^K B$, where $B =
\bigcup_j B_j$ and $B_j$ are given by (\ref{prebasis}), are
eigenvalues of $(X_{\Tk}, \R^d, \mu)$. We know $\phi^T$ is
invertible (it is expanding), $\rho$ is a linear isomorphism, and
$B$ is a basis of $\R^d$, hence we obtain a basis of $\R^d$
consisting of eigenvalues. Integer linear combinations of
eigenvalues are eigenvalues as well, so the set of eigenvalues of
$(X_{\Tk}, \R^d, \mu)$ is relatively dense in $\R^d$.
\end{proof}

The next lemma is essentially due to Robinson \cite{Robi} in a
more general case; we provide a proof for completeness.

\begin{lemma}  \label{eigenvalueImplyPisotFamily}
If $\bgamma$ is a non-zero eigenvalue of $(X_{\Tk}, \R^d, \mu)$,
then $Spec(\phi)$ is a Pisot family.
\end{lemma}

\begin{proof}
Let $\bx \in \Xi$. By Theorem~\ref{isomorphicImageOfC-contained}
we have $\bx = \rho(\sum_{j=1}^J p_j(\phi)\balpha_j)$ for some
polynomials $p_j \in \Z[x]$. Let $(\rho^T \bgamma)_j = p_j(\rho^T
\bgamma)$. We again use the linear injective map $\Fk:\,H_j\approx
\R^m \to \R^s \oplus \C^{2t}$ defined by (\ref{def-Fk}) and
obtain, using (\ref{eq-scalar}) and (\ref{eq-mult2}),
\begin{eqnarray}
\langle \phi^n\bx,\bgamma \rangle  & = & \sum_{j=1}^J \langle \phi^n p_{j}(\phi) \balpha_{j}, (\rho^T \bgamma)_{j} \rangle \nonumber \\
                                   & = & \sum_{j=1}^J \langle \Fk(\phi^n p_{j}(\phi)\balpha_{j}), \Fk((\rho^T\bgamma)_{j}) \rangle \nonumber \\
                           & = & \sum_{j=1}^J \left(\sum_{k=1}^s \lam_k^n p_j(\lam_k) z_{jk}\zeta_{jk} + \sum_{k=s+1}^{s+t} 2\re[\lam_k^n p_j(\lam_k) z_{j,2k-s-1}\zeta_{j,2k-s-1}]\right)\nonumber \\
                           & = & \sum_{k=1}^s c_k \lam_k^n + \sum_{k=s+1}^{s+t} (c_k \lam_k^n + \ov{c}_k \ov{\lam}_k^n), \nonumber
\end{eqnarray}
where $(z_{jk})_{k=1}^{s+2t} = \Fk(\balpha_j)$,
$(\zeta_{jk})_{k=1}^{s+2t} = \Fk((\rho^T \bgamma)_j)$, and $c_k$
are some complex numbers. By the assumption that $\bgamma$ is
an eigenvalue and \cite[Th.\,4.3]{soltil} we have \be
\label{dunduk} \dist\left(\sum_{k=1}^s c_k \lam_k^n + \sum_{k=s+1}^{s+t}
(c_k \lam_k^n + \ov{c}_k \ov{\lam}_k^n),\ \Z\right) = \dist( \langle \phi^n \bx,
\bgamma \rangle ,\Z) \stackrel{n \to \infty}{\longrightarrow} 0. \ee
Since $\Xi$ is relatively dense in $\R^d$ and $\bgamma\ne \b0$, we
can easily make sure that $\langle \bx,\bgamma \rangle\ne 0$, and
hence not all coefficients $c_k$ in (\ref{dunduk}) are equal to
zero. Then we can apply a theorem of Vijayaraghavan \cite[Th.\,4]{Vij} (a generalization of the classical result of Pisot and Vijayaraghavan)
to conclude that
$\Spec(\phi)$ is a Pisot family. (More precisely, we obtain that a
subset of $\Spec(\phi)$ is a Pisot family, but since all elements
of $\Spec(\phi)$ are conjugates and have modulus greater than one,
we get the claim.)
\end{proof}

\begin{theorem} \label{PisotFamily-relDenseEigenvalues}
Let $\Tk$ be a self-affine tiling of $\R^d$ with a diagonalizable
expansion map $\phi$. Suppose that all the eigenvalues of $\phi$
are algebraic conjugates with the same multiplicity. Then the
following are equivalent:
\begin{itemize}
\item[(i)] $Spec(\phi)$ forms a Pisot family;
\item[(ii)] the set of eigenvalues of $(X_{\Tk}, \R^d, \mu)$
    is relatively dense;
\item[(iii)] $(X_{\Tk}, \R^d, \mu)$ is not weakly mixing;
\item[(iv)] $\mathcal{C} = \{c(T) : T \in \Tk \}$ is a Meyer
    set.
\end{itemize}
\end{theorem}

\begin{proof}
(i) $\Rightarrow$ (ii) by Prop.\,\ref{set-of-eigenvalues}. \\
(ii) $\Rightarrow$ (iii) is obvious.\\
(iii)  $\Rightarrow$ (i) by
Lemma~\ref{eigenvalueImplyPisotFamily}.\\
(ii) $\Leftrightarrow$ (iv) by \cite[Th.\,4.14]{LS}.
\end{proof}

Theorem\,\ref{th-main} is contained in
Theorem~\ref{PisotFamily-relDenseEigenvalues}, so it is proved as
well.

\section{Structure of the control point set (proof of Theorem\,\ref{isomorphicImageOfC-contained})}
\label{Structure of control point set}

\noindent Now we make an isomorphic transformation $\tau$ of the
tiling $\Tk$ into another tiling whose control point set contains
$\balpha_{1}, \ldots, \balpha_{J}$ such that $\tau$ commutes
with $\phi$. This gives the structure of the control point set of $\Tk$
that we use in proving the main theorem in Section
\ref{Meyer-property}.

\medskip

In Corollary~\ref{basisWithAlphas} we showed that $W=\{\balpha_{1}, \ldots, \phi^{m-1} \balpha_{1},
\ldots, \balpha_{J}, \ldots, \phi^{m-1}\balpha_{J}\}$ is a basis for $\R^d$.
Since $\mathcal{C}$ is relatively dense in $\R^d$, for any
$\epsilon > 0$, for
every $j=1,\ldots,J$, there exists $\by_j\in \Ck$ such that
\[ \left\| \frac{\by_j}{\|\by_j\|} - \frac{\balpha_j}{\|\balpha_j\|} \right\| < \eps. \]
For $\eps>0$ sufficiently small, $(j-1)m +1, \dots, (jm)$-th entries of $\by_{j}$ are
all non-zero for any $1 \le j \le J$, and the set
 \[ Y := \{\by_1, \dots,
\phi^{m-1}\by_1, \dots, \by_J, \dots, \phi^{m-1}\by_J\} \subset
\mathcal{C}\] is a basis of $\R^d$.
We fix such a basis $Y$.

\begin{lemma} \label{isomorphism-changing-basis}
Let $\tau : \R^d \to \R^d$ be a linear map such that for each $1
\le j \le J$ and $0 \le k < m$,
\[\tau(\phi^k \by_j) = \phi^k \balpha_j \,.\]
Then $\tau$ is an isomorphism of $\R^d$ such that $\tau \phi =
\phi \tau$.
\end{lemma}

\begin{proof}
We first notice that $\tau$ is an isomorphism of $\R^d$, since $Y$
and $W$ are bases of $\R^d$.
In order to show that $\phi\tau(\bx) = \tau\phi(\bx),\ \bx\in \R^d$, it is enough to check this on the basis $Y$. For the vectors $\phi^k \by_j$,\ $0 \le k < m-1$, this
holds by definition, so we only need to consider $\phi^{m-1}\by_j$.
Let $p(t)$ be the characteristic polynomial of $\psi$. Then $p(\psi)=0$ by Cayley-Hamilton, and also $p(\phi)=0$, since $\phi$ is a
direct sum of $J$ copies of $\psi$. Thus
\begin{equation} \label{CH}
\phi^m =a_0I+\cdots + a_{m-1}\phi^{m-1}
\end{equation}
for some $a_0,\ldots,a_{m-1} \in \R$, hence
\begin{eqnarray*}
\phi\tau(\phi^{m-1}\by_j) = \phi^m \balpha_j & = & a_0 \balpha_j + \cdots + a_{m-1}\phi^{m-1}\balpha_j \\
                                             & = & a_0 \tau(\by_j)+\cdots + a_{m-1}\tau(\phi^{m-1}\by_j) = \tau(\phi^m \by_j),
\end{eqnarray*}
as desired.
\end{proof}

Let \be \label{new-substitution-tiling} \tau(\Tk):= \{\tau(T):\,
\tau(T)=(\tau(A), i), \ \mbox{where $T \in \Tk$ and $T=(A, i)$}
\}.\ee Note that $\tau(\Tk)$ is a fixed point of a primitive
substitution with the expansion map $\phi$. Indeed, $\om'(\tau(\Tk))=\tau(\Tk)$ where $\om'$ is defined by
\[ \om'(\tau(T_j)) = \{u + \tau(T_i): u \in \tau(\mathcal{D}_{ij}), i = 1, \dots, \kappa \} \ \ \
\mbox{for $j \le \kappa$,}\]
where
\[\phi\tau(A_j) = \bigcup_{i=1}^{\kappa} (\tau(\mathcal{D}_{ij}) + \tau(A_i)).\]
We define the tile map $\gamma': \tau(\Tk) \to \tau(\Tk)$ so that
for each $\Tk$-tile $T$,
\[\gamma'(\tau(T)) = \tau(\gamma(T)).\]
We define the control point for a tile $\tau(T) \in \tau(\Tk)$ by
\[\{c(\tau(T))\} = \bigcap_{n=0}^{\infty} \phi^{-n}((\gamma')^n \tau(\Tk)).\]
Then \[c(\tau(T)) = \tau c(T).\]

Applying the isomorphism $\tau$ commuting with $\phi$, we can reduce our problem to the case when
the control point set of the tiling contains $\balpha_{1}, \ldots, \balpha_{J}$. Thus, in
the rest of this section (except the last paragraph which proves
Theorem\,\ref{isomorphicImageOfC-contained}), we assume that
$\mathcal{C}$ contains $\balpha_{1}, \ldots, \balpha_{J}$.

\medskip

The following two propositions were obtained in
\cite{KS} in a special case. They are needed to get the structure of control
point set which we use in Section \ref{Meyer-property}. In the
appendix, we provide the proof, which is similar to that in \cite{KS}, for completeness.

In the next two propositions we {\bf do not} assume that all the eigenvalues of $\phi$ are conjugates and have the same multiplicity.
Let $G_\lam$ be the real $\phi$-invariant subspace of $\R^d$ corresponding to an eigenvalue $\lam \in \Spec(\phi)$.

\begin{prop} \label{prop-KenSol}
Let $\Ck$ be a set of control points for a self-affine tiling
$\Tk$ of $\R^d$ with an expansion map $\phi:\,\R^d \to \R^d$ which
is diagonalizable over $\C$. Let $\mathcal{C}_{\infty} =
\bigcup_{k=0}^{\infty} \phi^{-k} \mathcal{C}$ and let $\mathcal{D}$ be
a finitely generated $\Q[\phi]$-module containing
$\mathcal{C}_{\infty}$. Let $H$ be a vector space over $\R$ and
$A:\, H \to H$ be an expanding linear map, diagonalizable over $\C$.
Let $g: \Dk \to H$ be such that $g = A^{-1} \circ g \circ \phi$
and
\[g(y_1) - g(y_2) = g(y_1 - y_2) \ \ \ \mbox{for any $y_1, y_2 \in \Dk$}.\]
Let $f:= g|_{\Ck_\infty}:\,\Ck_\infty \to H$ be such that
for any $y_1, y_2 \in \mathcal{C}$,
\be \label{Lipschitz-linear-on-controlPointSet}
||f(y_1) - f(y_2)|| \le C ||y_1 - y_2 || \ \ \ \mbox{for some $C > 0$.}
\ee
Moreover, suppose that there exist $\gam>1$ and a norm
$\|\cdot\|$ in $H$ such that
\begin{equation} \label{eq-newk1}
\|A y\| \ge \gam\|y\|\ \ \mbox{for any}\ y\in H.
\end{equation}
Then the following hold:

{\bf (i)} The map $f$ is uniformly continuous on $\Ck_{\infty}$,
and hence extends by continuity to a map $f:\, \R^d\to H$
satisfying $f\circ \phi = A\circ f$.



{\bf (ii)} For any $\lam\in Spec(\phi)$ such that $|\lam| =
\gamma$, and any $a\in \R^d$,
\[ f|_{a + G_{\lam}} \ \mbox{is affine linear}. \]
\qed
\end{prop}

Let $P_{\lam}$ be the canonical projection of $\R^d$ to $G_{\lam}$
commuting with $\phi$, which exists by the diagonalizability
assumption on $\phi$. Denote by $G_{\lam}^\perp = (I -
P_{\lam})\R^d$ the complementary $\phi$-invariant subspace. We
consider the set $(I - P_{\lam})\Xi$, that is, the projection of
$\Xi$ to $G_{\lam}^\perp$ (recall that $\Xi$ is the set of translation vectors between two $\Tk$-tiles of the same type).
In some directions this projection may look like a lattice, i.e.\ be discrete.
We consider the directions in which this set is not discrete, and denote the span of
these directions by $G'$. We will prove that $f$ is affine linear on
all $G'$ slices.
More precisely, for any $\eps > 0$, define
\be \label{def-Geps}
G_{\eps} := \mbox{Span}_{\R} \left(B_{\eps} \cap (I - P_{\lam})\Xi \right)\ \ \ \mbox{and}\ \ \ G':= \bigcap_{\eps > 0} G_{\eps}.
\ee
Now let
\begin{eqnarray} \label{E-equal-to-E'-Elambda}
G := G' + G_{\lam}\,.
\end{eqnarray}
Note that $G$ is a subspace of $\R^d$ which is $\phi$-invariant,
because \[ \mbox{$\phi \Xi \subset \Xi \ \ $ and $\phi P_\lam = P_\lam \phi.$}\]

\begin{prop}  \label{prop-KenSol2}
Under the assumptions of Proposition~\ref{prop-KenSol}, $f|_{a + G}$ is affine linear for
any $a \in \R^d$. \qed
\end{prop}

\begin{lemma} \label{G-contain-all-eigenspaces}
Let all eigenvalues of $\phi$ be algebraic conjugates with the
same multiplicity. If $\lambda$ is the smallest in modulus
eigenvalue of $\phi$, then
\[G = G'+ G_{\lambda} = \R^d.\]
\end{lemma}

\begin{proof} This is proved in \cite{KS} (although not stated there
explicitly). Indeed, in the last part of \cite{KS}, labeled {\em
Conclusion of the proof of Theorem 3.1}, it is proved that the
subspace $G$ (denoted $E$ there) contains, for each conjugate of
$\lam$ greater or equal than $\lam$ in modulus, an eigenspace of
dimension at least $\dim(G_\lam)$. Note that in \cite{KS} the
setting is more general, of an arbitrary diagonalizable over $\C$
matrix $\phi$. In our case all eigenvalues are conjugates of the
same multiplicity, and $\lam$ is the smallest in modulus, hence
$G$ contains the entire $\R^d$.
\end{proof}

Since $\Tk$ has FLC, the $\Z$-module generated by $\Ck$, denoted
by $\langle \mathcal{C} \rangle_{\Z}$, is finitely generated. Let
$\{\bv_1, \dots, \bv_N\}$ be a generating set in $\mathcal{C}$.
For each $\bv_n$ with $1 \le n \le N$,
\[\bv_n = a_{11}^{(n)}
\balpha_1 + \cdots + a_{1m}^{(n)} \phi^{m -1} \balpha_1 + \cdots +
a_{J1}^{(n)} \balpha_J + \cdots + a_{Jm}^{(n)} \phi^{m -1}
\balpha_J\] where $a_{jk}^{(n)} \in \R$, $1 \le j \le J$, and $1
\le k \le m$. Thus
\[\mathcal{C} \subset \mathcal{D}:=\sum_{j=1}^J \sum_{k=1}^{m}
\sum_{n=1}^N \Q[\phi] (a_{jk}^{(n)} \balpha_j)= \sum_{j=1}^J \sum_{k=1}^{m}
\sum_{n=1}^N \Q[\phi_j] (a_{jk}^{(n)} \balpha_j).\]

\begin{lemma} \label{lem-field}
$\Q[\phi_{j}]$ is a field.
\end{lemma}

\begin{proof}
This is clearly a ring, so we just need to show that $p(\phi_j)$
has an inverse for $p \in \Q[x]$, if it is a non-zero matrix.
We need to use that all the eigenvalues of $\phi_j$ are conjugates
so they have the same irreducible polynomial $p(x)$. If $q(x) \in
\Z[x]$ is monic, such that $p(x)$ does not divide $q(x)$, then we
can find monic polynomials $h_1(x),h_2(x) \in \Z[x]$ such that
$h_1(x) p(x) + h_2(x) q(x) = 1$, which means that $h_2(\lam)$ is
the inverse of $q(\lam)$ for any eigenvalue $\lam$ of $\phi_j$.
\end{proof}

Let $\mathcal{D}_{j} = P_{j}(\mathcal{D})$, where $P_j$ is the
canonical projection of $\R^d$ onto $H_j$ as defined in
(\ref{def-projection}). Observe that $\mathcal{D}_{j}$ is a vector
space over the field $\Q[\phi_{j}]$, so we can write
\[\mathcal{D}_{j} = \bigoplus_{t=1}^{r_j} \Q[\phi_j] (a_{jt} \balpha_j)
= \bigoplus_{t=1}^{r_{j}} \Q[\phi] (a_{jt} \balpha_j) ,\]
where $a_{j1} =1$, $a_{jt} \in \R$ with $1 \le t \le r_{j}$, and
$\{a_{j1}, \dots, a_{jr_{j}}\}$ is linearly independent over $\Q$.
Note that $$\mathcal{D} = \mathcal{D}_{1} \bigoplus \cdots
\bigoplus \mathcal{D}_{J}.$$
We define $\Q[\phi]$-module homomorphisms
\[ \sigma_{j} : \Dk \rightarrow \Q[\phi] \balpha_{j}\]
such that \begin{eqnarray*} \left\{ \begin{array}{ll}
                      \sigma_{j}(a_{jt} \phi^n \balpha_{j}) = \phi^n \balpha_{j}
                      & \mbox{for any $ 1 \le t \le r_{j}$,  $n \in \Z_{\ge 0}$}\\
                      \sigma_{j}(a_{uw} \phi^n \balpha_{u}) =\bo
                      & \mbox{for any $u \neq j$,  $n \in
                      \Z_{\ge 0}$}\,.
                      \end{array}
                      \right. \end{eqnarray*}
Recall that $\mathcal{C}_{\infty} := \bigcup_{k=0}^{\infty} \phi^{-k}
\mathcal{C}$. Observe that $\Dk \supset \Ck_\infty$, since $\phi^{-1}$ is a rational linear combination of $\{I,\phi,\ldots,\phi^{m-1}\}$ by (\ref{CH}).
We define $\sigma_{j}':\mathcal{C}_{\infty} \to
\Q[\phi] \balpha_{j}$ to be the restriction of $\sig_j$, that is, $\sig_j': = \sig_j|_{\Ck_\infty}$.
\medskip

Using the same arguments as in \cite[Lemma 5.3]{sol-eigen} (which followed \cite{Thur}), we
obtain the next lemma.

\begin{lemma} \label{inequality-on-controlPoints}
For any $\xi, \xi' \in \Ck$,
\[\|\sigma'_{j}(\xi) - \sigma'_{j}(\xi')\| \le C ||\xi - \xi'|| \ \ \ \mbox{for some $C > 0$.}\]
\end{lemma}

\medskip

Now we use Prop.\,\ref{prop-KenSol},
Prop.\,\ref{prop-KenSol2} and
Lemma\,\ref{G-contain-all-eigenspaces} to prove
Theorem\,\ref{isomorphicImageOfC-contained}, and assume that all the assumptions of the latter hold. In addition, suppose that
the set of control points contains $\balpha_1,\ldots,\balpha_J$.
Fix $1 \le j \le
J$. We consider the maps $g = \sig_j:\,\Dk\to H_j$ and
$f = \sig'_{j}:\,\Ck_\infty \to H_{j}$, and let $A = \phi_{j} =
\phi|_{H_{j}}$.
Note
that (\ref{eq-newk1}) holds with $\gam$ equal to the smallest absolute value of
eigenvalues of $\phi_{j}$ (or $\phi$) and the norm defined as in (\ref{def-norm}).
Thus, all the hypotheses of Prop.\,\ref{prop-KenSol}, Prop.\,\ref{prop-KenSol2} and
Lemma\,\ref{G-contain-all-eigenspaces} are satisfied, and we obtain that
for each $1 \le j \le J$, the (extended) map $\sigma'_{j}$ is linear on
$\R^d$ and commutes with $\phi$.

\begin{lemma} \label{rho'-is-an-isomorphism}
The map $\rho':= \sigma'_{1} + \cdots + \sigma'_{J}$ is the identity map on
$\R^d$.
\end{lemma}

\begin{proof} Note that
\[\rho' : \R^d \rightarrow \R^d \ \ \mbox{is linear and $\rho' \phi = \phi \rho'$}
\]
and for any $1 \le j \le J$,
\begin{eqnarray*} \rho'(\balpha_j) &=& (\sigma'_{1} + \cdots +
\sigma'_{J})(\balpha_j) = 0 + \cdots + 0+ \sigma'_j (\balpha_j) + 0+ \cdots + 0=  \balpha_j \,.
\end{eqnarray*}
Since $W = \{\balpha_1, \dots, \phi^{m-1} \balpha_1, \dots,
\balpha_J, \dots, \phi^{m -1} \balpha_J \}$ is a basis of $\R^d$,
$\rho'$ is the identity map on $\R^d$.
\end{proof}

Now we do not assume that the control point set of $\Tk$ contains
$\balpha_{1}, \ldots, \balpha_{J}$ in order to prove
Theorem~\ref{isomorphicImageOfC-contained}. Instead, we apply the above
propositions and lemmas to $\tau(\mathcal{C})$.

\medskip

\noindent {\it{Proof of
Theorem\,\ref{isomorphicImageOfC-contained}}.} By Lemma~\ref{rho'-is-an-isomorphism}, for each
$\xi \in \mathcal{C}$,
\begin{eqnarray} \label{containment-of-C}
 \tau(\mathcal{\xi}) = \rho'(\tau(\mathcal{\xi})) &=& (\sigma'_{1} + \cdots +
 \sigma'_{J})(\tau(\xi)) = (\sigma_{1} + \cdots +
 \sigma_{J})(\tau(\xi)) \nonumber \\
& \in & \Q[\phi]\balpha_{1} + \cdots + \Q[\phi]\balpha_{J}\,.
\end{eqnarray}
Since $\mathcal{C}$ is finitely generated, we multiply (\ref{containment-of-C}) by a common
denominator $b \in \Z_+$ to get \be b
\cdot \tau(\mathcal{C}) \subset \Z[\phi] \balpha_{1} + \cdots +
\Z[\phi] \balpha_{J} \,. \nonumber \ee Let $\rho := \frac{1}{b}
\cdot \tau^{-1}$. Then
\[\mathcal{C} \subset \rho(\Z[\phi] \balpha_{1} + \cdots +
\Z[\phi] \balpha_{J})\] where $\rho$ is an isomorphism of $\R^d$ which
commutes with $\phi$. \qed

\section{Appendix}

\noindent We give the proofs of Prop.\,\ref{prop-KenSol} and
Prop.\,\ref{prop-KenSol2} after a sequence of auxiliary lemmas.
The arguments are similar to those in \cite{KS}, but we
present them in a more general form for our purposes.

Denote by $B_{R}(a)$ the open ball of radius $R$ centered at $a$ and let $B_R:= B_R(0)$. We will also write $\ov{B}_R(a)$ for the closure of $B_R(a)$.
Let $r = r(\Tk)
>0$ be such that for every $a\in \R^d$ the ball $\ov{B}_r(a)$ is covered by a tile containing $a$ (which need not be unique) and its neighbors.
Let $\lam_{\max}$ be the largest eigenvalue of $\phi$.

\begin{lemma}
The function $f$ is uniformly continuous on $\mathcal{C}_{\infty}$.
\end{lemma}

\begin{proof}
This is very similar to \cite[Lem.\,3.4]{KS}.
It is enough to show that
\be \label{Holder-continuity} \xi_1, \xi_2 \in \Ck_\infty,\ \ \|\xi_1 - \xi_2\|\le r\ \ \Longrightarrow\ \ \|f(\xi_1) - f(\xi_2)\|\le L \|\xi_1-\xi_2\|^\alpha,
\ee
for $\alpha =
\frac{\log\gamma}{\log|\lambda_{\max}|}$ and some $L
> 0$ (that is, $f$ is H\"{o}lder continuous on $\Ck_\infty$).

Let $\xi_1, \xi_2 \in \mathcal{C}_{\infty}$ satisfy  $||\xi_1 -
\xi_2|| = \delta \le r$. Then there exist $y_1, y_2 \in
\mathcal{C}$ such that $\phi^{-s} y_1 = \xi_1$ and $\phi^{-s} y_2
= \xi_2$ for some $s \in \Z_{\ge 0}$. We choose the smallest
$l \in \Z_{\ge 0}$ such that \[ \phi^s B_{\delta}(\phi^{-s} y_1)
\subset \phi^l B_r(\phi^{-l} y_1),\] which is equivalent to $\phi^{s-l} B_{\delta}(\phi^{-s} y_1) \subset B_r(\phi^{-l} y_1)$.
Since $\delta \le r$, we have $l \le s$ and hence $l$ is the smallest integer satisfying
$$
|\lam_{\max}|^{s-l} \delta \le r.
$$
Thus, \be \label{randdelta-inequality} |\lam_{\max}|^{s-l} >
\frac{r}{\delta} |\lam_{\max}|^{-1}. \ee Observe that $y_2 \in
\phi^s \ov{B}_\delta(\phi^{-s} y_1)\subset \phi^l
\ov{B}_r(\phi^{-l} y_1)$, hence $\phi^{-l} y_1$ and $\phi^{-l}
y_2$ are in the same or in the neighboring tiles of $\Tk$ by the
choice of $r$. It is shown in the course of the proof of
\cite[Lem.\,3.4]{KS} that we can write $y_1-y_2 = \sum_{h=1}^l
\phi^h w_h$,  where $w_h \in W$ for some finite set $W \subset
\phi^{-1} \Xi$ which depends only on the tiling $\Tk$ (a similar
statement, but without precise value of $l$ is proved in
\cite[Lemma\,4.5]{LS}). So
\begin{eqnarray*}
||f(\xi_1) - f(\xi_2)|| & = &|| f (\phi^{-s}
y_1) - f (\phi^{-s} y_2)|| = ||A^{-s} g (y_1) -
A^{-s} g (y_2)|| \\
& = & ||A^{-s} (g (y_1) - g (y_2))|| = ||A^{-s} g (y_1 - y_2)|| \\
& = & || A^{-s} g (\sum_{h=1}^l  \phi^h w_h)|| =  ||A^{-s} \sum_{h=1}^l A^h g(w_h)|| \\
& = & || \sum_{h=1}^l A^{h-s} g(w_h)|| \le \sum_{h=1}^l \gamma^{h-s}||g(w_h)|| \\
& \le &  L' \gamma^{l-s}\,,
\end{eqnarray*}
for some $L'> 0$ independent of $l$. Notice that
$\gamma^{l-s} = (|\lambda_{\max}|^{l -s})^{\alpha}$, where
$\alpha = \frac{\log\gamma}{\log|\lambda_{\max}|}$. Thus
\begin{eqnarray*}
||f(\xi_1) - f(\xi_2)|| &\le& L'
(|\lambda_{\max}|^{l-s})^{\alpha} \\
& < & L'\left( \frac{|\lambda_{\max}|}{r} \delta \right)^{\alpha} \ \ \ \mbox{by} \ (\ref{randdelta-inequality})\\
& = & L ||\xi_1 - \xi_2||^{\alpha} \ \ \ \mbox{where} \ L:= L'
\left( \frac{|\lambda_{\max}|}{r} \right)^{\alpha}\,,
\end{eqnarray*}
and (\ref{Holder-continuity}) is proved.
\end{proof}

\medskip

Since $\mathcal{C}_{\infty}$ is dense in $\R^d$, we can
extend $f$ to a map $f : \R^d \to H$ by continuity, and moreover,
$$
f\circ \phi = A \circ f
$$
(we denote the extended function by the same symbol $f$).
This proves part (i) of Prop.~\ref{prop-KenSol}.

\begin{lemma} \label{depend-only-on-tile-type}
Let $T$ and $T + z$ be tiles in $\Tk$. Then \be
\label{translational-on-same-tiles} f(\xi +z) = f(\xi) + g(z) \ \
\mbox{for any} \ \xi \in \supp(T).\ee
\end{lemma}

\begin{proof} It is enough to show that (\ref{translational-on-same-tiles})
holds for a dense subset of $\supp(T)$, namely, $\Ck_\infty \cap \supp(T)$. Suppose that $\xi =
\phi^{-k} c(S)$, where $S \in \omega^k (T)$. Note that \be
\label{in-a-supertile} S + \phi^k z \in \omega^k (T + z) \subset
\Tk. \ee Then
\begin{eqnarray*}
f(\xi + z) &=& f(\phi^{-k} c(S) + z) =
f(\phi^{-k}(c(S) + \phi^k  z)) \\
&=& f\phi^{-k}(c(S + \phi^k z)) = A^{-k} g (c(S + \phi^k  z)) \\
&=& A^{-k}( g (c(S))+ g(\phi^k z)) = A^{-k} g (c(S)) + A^{-k}g \phi^k(z) \\
&=& f(\phi^{-k} c(S)) + g(z) = f(\xi) + g(z) \,.
\end{eqnarray*}
\end{proof}

Recall that $A$ is diagonalizable over $\C$. For $\theta \in \Spec(A)$ let $p_\theta: H\to H$ be the canonical projection onto the real $A$-invariant subspace for
$A$ corresponding to $\theta$, so that we have
$$
I_H = \sum_{\theta \in \Spec(A)} p_\theta.
$$
Define
\[f_\theta = p_{\theta} \circ f, \ \ \ \theta\in \Spec(A).\]
Note that \be \label{sigmaEqualtoSmallSigmas} f= \sum_{\theta \in \Spec(A)} f_\theta.\ee
Suppose that $\lam \in Spec(\phi)$ satisfies
$|\lam| = \gamma$.

\begin{lemma} \label{piij-Lipschitz-constant}
For $\theta \in Spec(A)$ and $a \in \R^d$,
\[ \left\{ \begin{array}{ll}
         f_\theta|_{a + G_{\lambda}} \ \mbox{is Lipschitz} \ &
         \mbox{if $ |\theta| = |\lambda|$}\,; \\
         f_\theta|_{a + G_{\lambda}} \ \mbox{is constant} \ &
         \mbox{if $ |\theta| > |\lambda|$} \,.
         \end{array}
  \right.
 \]
 Moreover, the Lipschitz constant is uniform in $a\in \R^d$ (equal to $C$ from (\ref{Lipschitz-linear-on-controlPointSet})).
\end{lemma}

\begin{proof} Let $\xi_1, \xi_2 \in a + G_{\lambda}$ for some $a \in \R^d$.
For any $l \in \Z_+$, using the norm in $H$ analogous to that in (\ref{def-norm}), so that $\|A \circ p_\theta(\bx) \| = |\theta|\,\|p_\theta(\bx)\|$, we obtain
\begin{eqnarray} \label{pi-ij-inequality-on-control-points}
||f_\theta(\xi_1) - f_\theta(\xi_2)|| &=&
||(p_{\theta} \circ f) \phi^{-l} (\phi^l \xi_1) -  (p_{\theta} \circ f) \phi^{-l} (\phi^l \xi_2)|| \nonumber\\
& = & ||p_{\theta}(A^{-l} f(\phi^l \xi_1) - A^{-l} f(\phi^l \xi_2))||\nonumber \\
& = & |\theta|^{-l}|| p_{\theta} (f(\phi^l \xi_1) -
f(\phi^l \xi_2))|| \nonumber \\
 & \le & |\theta|^{-l}||
f(\phi^l \xi_1) - f(\phi^l \xi_2)||.
\end{eqnarray}
Note that there exist $y_1, y_2 \in \mathcal{C}$ such that
\[||\phi^l \xi_1 - y_1|| <
\delta_1 \ \ \mbox{and} \ ||\phi^l \xi_2 - y_2|| < \delta_1 \]
for some fixed $\delta_1 >0$. Since $f$ is uniformly continuous,
\[ ||f(\phi^l \xi_1) - f(y_1)|| <
\delta_2 \ \  \mbox{and} \ ||f(\phi^l \xi_2) -
f(y_2)|| < \delta_2 \] for some fixed $\delta_2 > 0$.
By the assumption on $f$,  we have $||f(y_1) - f(y_2)|| \le C ||y_1 -
y_2||$. Thus
\begin{eqnarray} \label{pi-ij-controlPoints-Lipschitz}
||f(\phi^l \xi_1) - f(\phi^l \xi_2)|| & < &
||f(y_1)- f(y_2)|| + 2 \delta_2 \nonumber \\
& \le & C ||y_1 - y_2|| + 2 \delta_2 \nonumber \\
& < & C (||\phi^l \xi_1 - \phi^l \xi_2|| + 2\delta_1) + 2
\delta_2.
\end{eqnarray}
Applying (\ref{pi-ij-controlPoints-Lipschitz}) to
(\ref{pi-ij-inequality-on-control-points}), we obtain that for any
$l \in \Z_{+}$
\begin{eqnarray*}
\lefteqn{||f_\theta(\xi_1) - f_\theta(\xi_2)||} \\
& < & |\theta|^{-l}
\left( C ||\phi^l \xi_1 - \phi^l \xi_2|| + 2C \delta_1 + 2 \delta_2 \right) \\
& = & \frac{C |\lambda|^l}{|\theta|^l} ||\xi_1 -
\xi_2|| + \frac{1}{|\theta|^l}(2C \delta_1 + 2\delta_2)\,.
\end{eqnarray*}
Thus if $|\theta| = |\lambda|$, we have that $f_\theta|_{a+ G_{\lambda}}$ is
Lipschitz with a uniform  Lipschitz constant $C$, and if $|\theta|
> |\lambda|$, we have that $f_\theta|_{a+ G_{\lambda}}$ is
constant.
\end{proof}

\begin{remark} \label{subspace-generated-by-eigensp}
{\em First note that $|\lam| = \gamma \le {\mmin}\{|\theta| :
\theta \in Spec(A)\}$ by (\ref{eq-newk1}). The last lemma implies that for any $\xi \in \R^d$ and $w \in
G_{\lam}$, the vector $f(\xi + w) - f(\xi)$ is in the subspace generated by
eigenspaces of $A$ corresponding to eigenvalues $\theta$ for
which $|\theta| = |\lam|$. We make use of this observation to
show (\ref{phi-phi}) in Lemma~\ref{pi-i-on-E-lamimin} below.}
\end{remark}

From Lemma~\ref{piij-Lipschitz-constant} and
(\ref{sigmaEqualtoSmallSigmas}), we get the following corollary.

\begin{cor} \label{pi-i-Lipschitz}
$f|_{a + G_{\lambda}}$ is Lipschitz for any $a \in \R^d$.
\end{cor}

We now prove furthermore that $f$ is affine linear on $
G_{\lambda}$ slices of $\R^d$.

\begin{lemma} \label{pi-i-on-E-lamimin}
$f|_{a + G_{\lambda}}$ is affine linear for any $a \in \R^d$.
\end{lemma}

\begin{proof} This is analogous to \cite[Lem.\,3.7]{KS}, but in some places the presentation is sketchy, so we provide complete details for the
readers' convenience.

Since $f|_{a+ G_{\lambda}}$ is Lipschitz for any $a \in \R^d$,
it is a.e.\ differentiable by Rademacher's theorem, and hence $f$ is differentiable in the direction of $G_{\lambda}$ a.e. in $
\R^d$, by Fubini's theorem. Let \[ D(z)u = \lim_{t \to 0}
\frac{f(z + tu) - f(z)}{t} \ \ \ \mbox{for} \ u \in G_{\lambda} \
\mbox{and} \ z \in \R^d .\] The limit exists a.e.\ $z \in \R^d$ and
for all $u \in G_{\lam}$, and $D(z)u$ is a linear transformation
in $u$ (from $G_{\lambda}$ to $H$). Moreover $D(z)$ is a
measurable function of $z$, being a limit of continuous functions. By the definition of total derivative,
\[ \lim_{n \to \infty} F_n(z)=0\ \ \
\mbox{for a.e.} \ z \in \R^d,\ \ \mbox{where}\ \
\ F_{n}(z) = \sup_{\stackrel{u
\in G_{\lambda}}{0 < \|u\| < \frac{1}{n}}} \frac{\|f(z + u) - f(z)
- D(z)u \|}{\|u\|}\,.\]
By Egorov's theorem, $\{F_n\}$
converges uniformly on a set of positive measure. This implies that
there exists a sequence of positive integers $N_l \uparrow
\infty$ such that
\be \label{define-Omega}
\Omega := \left\{\xi \in \R^d  :\ \forall\,l\ge 1,\ \forall\, u \in B_{1/N_l}\cap G_l,\
\frac{\|f(\xi + u) - f(\xi) - D(\xi)u\|}{\|u\|} < \frac{1}{l} \right\}
\ee has positive Lebesgue measure.

Our goal is proving that $\Omega$ has full Lebesgue measure. The
argument is based on a kind of ``ergodicity''. First observe from
Lemma\,\ref{depend-only-on-tile-type} that $\Omega$ is ``piecewise
translation-invariant'' in the following sense: \be
\label{eq-trans} (T\in \Tk,\ T+x\in \Tk,\ \xi\in \Omega\cap
\supp(T))\ \Longrightarrow\ \xi + x \in \Omega. \ee Second,
$\Omega$ is forward invariant under the expansion map $\phi$.
Indeed, let $\xi \in \Omega$ and $u \in \phi(B_{1/N_l}) \cap
G_\lam$. Then
\begin{eqnarray}
\lefteqn{\|f(\phi\xi +u) - f(\phi \xi) - A D(\xi) \phi^{-1} u\|} \nonumber \\ & = & \|A(f(\xi+\phi^{-1}u) - f(\xi) - D(\xi) \phi^{-1}u)\|  \nonumber \\
                                                       & = & |\lam|\|f(\xi + \phi^{-1}u) - f(\xi) - D(\xi) \phi^{-1}u)\|  \ \ \
                             \ \ \ \   \mbox{by Remark~\ref{subspace-generated-by-eigensp}} \nonumber \\
                               & < & |\lam| \frac{\|\phi^{-1}u\|}{l} = \frac{|\lam|}{|\lam|} \frac{\|u\|}{l} =
                                \frac{\|u\|}{l}\,. \label{phi-phi}
 \end{eqnarray}
This implies that $D(\phi \xi)$  exists and equals $AD(\xi) \phi^{-1}$, and since $\phi(B_{1/N_l}) \supset B_{|\lam|/N_l} \supset B_{1/N_l}$, we also obtain that
$$
\phi(\Omega) \subset \Omega.
$$

We will need a version of the Lebesgue-Vitali density theorem where the differentiation basis is the collection of sets of the form $\phi^{-l}B_1,\ l\ge 0$, and
their translates. It is well-known that such sets form a density basis, see \cite[pp.\,8--13]{Stein}. Let $y$ be a density point of $\Omega$ with respect to this density
basis. Then
\[m(\Omega \cap \phi^{-l}B_1 (\phi^l y)) \ge (1 - \eps_l) m(\phi^{-l}B_1) \ \ \ \mbox{for some} \ \eps_l \to 0, \]
where $m$ denotes the Lebesgue measure.
Note that \be m(\Omega \cap B_1 (\phi^l y)) & \ge & m(\phi^l
\Omega \cap B_1 (\phi^l
y)) \nonumber \\
& = & |\det \phi|^l m(\Omega \cap \phi^{-l} B_1 (\phi^l y)) \nonumber \\
& \ge & |\det \phi|^l (1 - \eps_l) m(\phi^{-l}B_1)\nonumber \\
& = & (1 - \eps_l) m(B_1). \nonumber \ee By FLC and repetitivity, there exists $R>0$ such that $B_R$ contains equivalence classes of all the patches $[B_1(\phi^l y)]^\Tk$.
Then for any $l \in \Z_+$, there exists $y_l \in B_R$ such that
$$
[B_1(y_l)]^{\Tk} = [B_1(\phi^l y)]^\Tk + (y_l - \phi^l y).
$$
By (\ref{eq-trans}), we have
$$
m(\Omega \cap B_1(y_l))  \ge (1 - \eps_l) m(B_1),
$$
hence $m(\Omega \cap B_1(y')) = m(B_1)$ for any limit point $y'$ of the sequence $\{y_l\}$. We have shown that $\Omega$ is a set of full measure in $B_1(y')$.
But then it is also a set of full measure in $\phi^k B_1(y')$ for $k \ge 1$. By the repetivity of $\Tk$, using (\ref{eq-trans}), we obtain that
$\Omega$ has full measure in $\R^d$.

Choose $n_l \in \Z_+$ so that
$|\lam|^{n_l} > N_l$. Repeating the argument of (\ref{phi-phi}) we obtain
\be \xi \in \phi^{n_l} \Omega &\Rightarrow & ||f(\xi + v) - f(\xi) - D(\xi) v|| \le \frac{||v||}{l} \nonumber \\
&& \ \ \ \mbox{for all} \ v \in \phi^{n_l} \left( B_{1/{N_l}} \cap
G_{\lam} \right) \supset B_{|\lam|^{n_l}/N_l} \cap G_\lam \supset
B_1 \cap G_{\lam} \, . \nonumber \ee Thus $f(\xi + v) = f(\xi) +
D(\xi) v$ for any $\xi \in \bigcap_{l=1}^{\infty}
\phi^{n_l}\Omega$ and $v \in B_1 \cap G_{\lam}$. Note that
$\bigcap_{l=1}^{\infty} \phi^{n_l} \Omega$ has full measure, hence
it is dense in $\R^d$. So for any $\xi \in \R^d$, we can find a
sequence $\{\xi_j\} \subset \bigcap_{l=1}^{\infty} \phi^{n_l}
\Omega$ such that $\xi_j \to \xi$. Since $f|_{\xi_j + G_\lam}$ is
Lipschitz with a uniform Lipschitz constant $C$, the derivatives
$D(\xi_j)$ are uniformly bounded, and we can assume that
$D(\xi_j)$ converges to some linear transformation $D_\xi$ by
passing to a subsequence. Then we can let $j\to \infty$ to obtain
$$
f(\xi + v) = f(\xi) + D_\xi v\ \ \ \mbox{for all}\ v\in B_1 \cap G_\lam.
$$
Since this holds for every point in $\R^d$, we obtain that $D_\xi = D_{\xi'}$ for any $\xi, \xi' \in \R^d$ with $\xi - \xi' \in G_{\lam}$, and
$f|_{\xi+G_{\lam}}$ is affine linear for any $\xi \in \R^d$.
\end{proof}

This concludes the proof of Proposition~\ref{prop-KenSol}.

\medskip

Recall (\ref{def-Geps}) that $G' = \bigcap_{\eps>0} G_\eps$ and $G_\eps = \Span_\R(B_\eps \cap (I-P_\lam)\Xi)$.

\begin{lemma} \label{lem-G'}
There exists $\eps>0$ such that $G' = G_{\eps'}$ for every $0 < \eps'\le \eps$, and moreover,
$$
G' = G'':=\Span_\R(B_{\eps'} \cap (I-P_\lam)(\Ck_1- \Ck_1))\ \ \mbox{for all}\ 0 < \eps'\le \eps.
$$
\end{lemma}

\begin{proof}
Observe that $G_{\eps'} \subset G_\eps$ for $\eps' < \eps$. These are finite-dimensional subspaces over $\R$, hence they must stabilize, which yields the first claim.

To prove the second claim, we just need to show $G' \subset G''$ since $\Ck_1-\Ck_1 \subset \Xi$.
There exists $k \in \Z_+$ such that $\phi^k\Xi \subset \Ck_1-\Ck_1$ (just choose $k$ such that
$\om^k(T_1)$ contains tiles of all types). Then
$$
G' = \phi^k G' = \phi^k G_{\eps'/\|\phi\|^k} \subset \Span_\R (B_{\eps'} \cap (I-P_\lam) \phi^k \Xi) \subset G'',
$$
as desired.
\end{proof}

%

\medskip

\noindent {\em Proof of Proposition~\ref{prop-KenSol2}.} This is similar to \cite[Lem.\,3.8]{KS}, but again, there are some differences, and we provide more details here.

For any $z \in \R^d$ and $\xi \in G$, we have $f(z+ P_{\lam}\xi) = f(z)
+ E(z)P_{\lam} \xi$ by Lemma \ref{pi-i-on-E-lamimin}. Since $f$ is
uniformly continuous, $E(z)$ is independent of the choice of $z$.
So for $\xi \in G$,
\begin{eqnarray} f(a + \xi) &=&
f(a + (I - P_{\lam})\xi + P_{\lam} \xi) \nonumber \\
&=& f(a + (I - P_{\lam})\xi) + E \, P_{\lam} \xi \label{pumba}
\,,
\end{eqnarray} for some fixed linear transformation $E:\,G_\lam\to H$. Thus we only need to prove that
$f$ is affine linear on all $G'$-slices.
Let $T$ be a tile of type 1 in $\Tk$. For any $a \in
(\supp(T))^{\circ}$, choose $r
> 0$ such that $B_r(a) \subset (\supp(T))^{\circ}$. We will show that
\be \label{Jensen-functional-equation} f\left(\frac{\zeta_1 +
\zeta_2}{2}\right) = \frac{f(\zeta_1) + f(\zeta_2)}{2} \ \ \
\mbox{for all} \ \zeta_1, \zeta_2 \in B_r(a) \cap (a + G')\,. \ee
In other words, $f|_{a+G'}$ satisfies the so-called Jensen
functional equation, and since $f$ is continuous, this will imply
that $f|_{a+G'}$ is locally affine linear, see
\cite[2.1.4]{Aczel}. By expanding (using $\phi$-invariance) and
translating (using (\ref{eq-trans})), we will then conclude that
$f|_{a+G'}$ is affine linear for all $a \in \R^d$.

Now we show (\ref{Jensen-functional-equation}). By Lemma~\ref{lem-G'}, for any
$\eps' > 0$ with $\eps' \le \eps$, there exists a basis $\{y_1,
\cdots, y_s \}$ of $G'$ such that for each $1 \le j \le s$, $y_j
\in B_{\eps'}$ and $y_j = (I - P_{\lam}) z_j$ for some $z_j \in
\mathcal{C}_1 - \mathcal{C}_1$. Let
$\zeta_1, \zeta_2 \in B_r(a) \cap (a + G')$ and fix small $\eps'
> 0$ such that $\eps' \le \eps$ and
\[ \eps' < \frac{(r-\mbox{max}\{||\zeta_1||, ||\zeta_2||\})}{4s}.\]  We consider the lattice
generated by the $y_j$'s in $G'$. It defines a grid with
grid cells of diameter less than $s \max_j||y_j|| \le s \eps'$.
Thus there exist $b_j \in \Z$, $1 \le j \le s$, such that
\[||\sum_{j=1}^s b_j y_j - \frac{(\zeta_2 - \zeta_1)}{2}|| < s
\eps'.\] Let $\tilde{\zeta} : = \zeta_1 + \sum_{j=1}^s b_j y_j$, so that $$\|\frac{\zeta_1+\zeta_2}{2}-\tilde{\zeta}\|< s\eps'.$$
Translate our grid in such a way that $\zeta_1$ is the
origin and consider a `grid geodesic' connecting $\zeta_1$ to
$\tilde{\zeta} = \zeta_1 + \sum_{j=1}^s b_j y_j $ in the $s
\eps'$-tube around the line segment $[\zeta_1, \tilde{\zeta}]$. By
the choice of $\eps'$, this `grid geodesic' is contained in
$B_r(a) \cap (a+ G')$. It is a sequence of points $\xi_1 =
\zeta_1, \xi_2, \cdots, \xi_L = \tilde{\zeta}$, where $\xi_{l+1} -
\xi_l = y_{t(l)}$, $y_{t(l)} = (I - P_{\lam})z_{t(l)}$, $z_{t(l)}
\in \mathcal{C}_1 - \mathcal{C}_1$, and $L = \sum_{j=1}^s |b_j|$.
For each $z_{t(l)} \in \mathcal{C}_1 - \mathcal{C}_1$, there
exists a tile $S$ of type 1 such that $S + z_{t(l)} \in \Tk$. By
Lemma~\ref{depend-only-on-tile-type} we have
$f(\eta + z_{t(l)}) = f(\eta) + g (z_{t(l)})$ for any $\eta \in
\supp(S)$. In view of (\ref{pumba}),
\be \label{linearity-on-tile} f(\eta +y_{t(l)}) =
f(\eta) + g(z_{t(l)}) - E \, P_{\lam} (z_{t(l)}) \,. \ee
Since $\xi_l, \xi_{l+1}\in \supp(T)$ which is of type 1, using
Lemma~\ref{depend-only-on-tile-type} again we obtain
\begin{eqnarray*}
f(\xi_{l+1}) - f(\xi_{l}) & = & f(\xi_{l} + y_{t(l)}) - f(\xi_{l}) \\
                          & = & f(\eta +y_{t(l)}) - f(\eta) = g(z_{t(l)}) -
              E \, P_{\lam}(z_{t(l)}).
\end{eqnarray*}
Then \be \label{puma1} f(\tilde{\zeta}) - f(\zeta_1) =
\sum_{l=1}^L (g(z_{t(l)}) - E \, P_{\lam}(z_{t(l)})).\ee Note that
$\zeta_2 - (\zeta_2 + \zeta_1 - \tilde{\zeta}) = \sum_{j=1}^s b_j
y_j$. The point $\zeta_2 + \zeta_1 - \tilde{\zeta}$ is symmetric
to $\tilde{\zeta}$ with respect to $\frac{\zeta_1+\zeta_2}{2}$, so
it is also within $s\eps'$ of $\frac{\zeta_1+\zeta_2}{2}$. The
grid geodesic which connected $\zeta_1$ to $\tilde{\zeta}$,
translated by $\zeta_2-\tilde{\zeta}$, connects $\zeta_2 + \zeta_1
- \tilde{\zeta}$ to $\zeta_2$ inside $B_r(a) \cap (a+ G')$. Thus,
we obtain, repeating the argument above, that \be \label{puma2}
f(\zeta_2) - f(\zeta_2 + \zeta_1 - \tilde{\zeta}) = \sum_{l=1}^L
(g(z_{t(l)}) - E \, P_{\lam}(z_{t(l)})).\ee Since $||\tilde{\zeta}
- \frac{\zeta_2 + \zeta_1}{2}|| < s \eps'$, by uniform continuity
\[\mbox{max}\{||f(\tilde{\zeta}) - f(\frac{\zeta_2 + \zeta_1}{2})||,
||f(\zeta_2 + \zeta_1 - \tilde{\zeta}) -
f(\frac{\zeta_2 + \zeta_1}{2})||\} < \delta(\eps')\]
where $\delta(\eps')\to 0$ as $\eps'\to 0$. Combining this with (\ref{puma1}) and (\ref{puma2}) yields
(\ref{Jensen-functional-equation}), as desired.
\qed

\medskip

This completes the proof of Theorem~\ref{isomorphicImageOfC-contained}.


\begin{thebibliography}{99}

\bibitem{Aczel} J. Acz\'el, {\em Lectures on Functional Equations and Their Applications}.
Mathematics in Science and Engineering, Vol. 19, Academic Press, New York-London 1966.

\bibitem{AL} S. Akiyama and J.-Y. Lee, An efficient algorithm for determining whether
substitution tiling has pure point spectrum, Preprint.

\bibitem{AP} J. Andersen and I. Putnam,
Topological invariants for substitution tilings and their associated
$C^*$-algebras,
{\em Ergodic Th.\ Dynam.\ Sys.} {\bf 18} (1998), no.\ 3, 509--537.

\bibitem{BL}
M. Baake and D. Lenz, Dynamical systems on translation bounded
measures : pure point dynamical and diffraction spectra, {\em
Ergodic Theory Dynam. Systems} {\bf 24} (2004), 1867--1893.

\bibitem{BLM} M. Baake, D. Lenz, and R. V. Moody,
Characterization of model sets by dynamical systems, {\em Ergodic
Theory Dynam. Systems} {\bf 27} (2007), no. 2, 341--382.

\bibitem{BG} R. Benedetti and J.-M. Gambaudo,
On the dynamics of $\mathbb G$-solenoids. Applications to Delone sets,
{\em  Ergodic Theory Dynam.\ Sys.} {\bf 23} (2003), no.\ 3, 673--691.

\bibitem{CS} A. Clark and L. Sadun, When shape matters: deformations of tiling spaces, {\em  Ergodic Theory Dynam.\ Systems}  {\bf 26}  (2006),  no.\ 1, 69--86.

\bibitem{Danzer} L. Danzer,
Inflation species of planar tilings which are not of locally finite complexity,
{\em Proc.\ Steklov Inst.\ Math.} {\bf 239} (2002), no.\ 4, 108--116.

\bibitem{Dw} S. Dworkin, Spectral theory and $X$-ray diffraction,
{\em J.\ Math.\ Phys.} {\bf 34} (1993), 2965--2967.

\bibitem{natalie} N. P. Frank, A primer of substitution tilings of the Euclidean plane, {\em  Expo. Math.} {\bf  26}  (2008),  no.\ 4,
295--326.

\bibitem{FraRob} N. P. Frank and E. A. Robinson, Jr.,
Generalized $\beta$-expansions, substitution tilings, and local finiteness,
{\em Trans.\ Amer.\ Math.\ Soc.}  {\bf 360}  (2008),  no.\ 3, 1163--1177.

\bibitem{Gamb} J.-M. Gambaudo, A note on tilings and translation surfaces, {\em  Ergodic Theory Dynam.\ Systems} {\bf 26}  (2006),  no.\ 1, 179--188.

\bibitem{Gouere}  J.-B. Gou\'{e}r\'{e},
Quasicrystals and almost periodicity.  {\em Comm.\ Math.\ Phys.}  {\bf 255}  (2005),  no.\ 3, 655--681.

\bibitem{HS} M. Hirsch and S. Smale, {\em Differential Equations, Dynamical Systems, and Linear Algebra}, Academic Press, 1974.

\bibitem{HRS}
C. Holton, C. Radin, and L. Sadun, Conjugacies for tiling dynamical systems,
{\em Comm.\ Math.\ Phys.} {\bf 254} (2005), no.\ 2, 343--359.

\bibitem{Kel} J. Kellendonk, Pattern equivariant functions, deformations and equivalence of tiling spaces, {\em  Ergodic Theory Dynam.
\ Systems} {\bf   28}  (2008),  no.\ 4, 1153--1176.

\bibitem{Ken} R. Kenyon, Inflationary similarity-tilings,
{\em Comment.\ Math.\ Helv.} {\bf 69} (1994), 169--198.

\bibitem{Kenyon.construction} R. Kenyon,
The construction of self-similar tilings,
{\em Geom.\ Funct.\ Anal.} {\bf  6} (1996), no.\ 3, 471--488.

\bibitem{Ken.thesis} R. Kenyon, Ph.D.\ Thesis, Princeton University,
1990.

\bibitem{KS} R. Kenyon and B. Solomyak, On the characterization of
expansion maps for self-affine tilings, {\em Discrete Comp. Geom.}
Online, 2009.


\bibitem{Lag} J. C. Lagarias,
Mathematical quasicrystals and the problem of diffraction, in {\em Directions in Mathematical Quasicrystals},
M.~Baake and R.~V.~ Moody, eds., CRM Monograph Series, Vol. 13, AMS, Providence, RI, 2000, pp. 61--93,

\bibitem{lawa} J. C. Lagarias and Y. Wang,
Substitution Delone sets, {\em Discrete Comput. Geom.} {\bf 29}
(2003), 175--209.
\bibitem{Lang} S. Lang, Algebra, Addison Wesley, second edition,
1984.

\bibitem{Lee} J.-Y. Lee, Substitution Delone Sets with Pure Point Spectrum are Inter Model Sets,
{\em  J. Geom. Phys.} {\bf 57} (2007) 2263-2285.

\bibitem{LMS} J.-Y. Lee, R. V. Moody, and B. Solomyak,
Pure Point Dynamical and Diffraction Spectra, {\em Ann. Henri Poincar{\'e}}
{\bf 3} (2002), 1003--1018.

\bibitem{LS} J.-Y Lee and B. Solomyak, Pure point diffractive substitution
Delone sets have the Meyer property, {\em Discrete Comp. Geom.} {\bf 39} (2008), 319--338.

\bibitem{LMS2} J.-Y. Lee, R. V. Moody, and B. Solomyak,
Consequences of Pure Point Diffraction Spectra for Multiset
Substitution Systems, {\em Discrete Comp. Geom.} {\bf 29} (2003),
525--560.

\bibitem{Lind}
D. Lind,
The entropies of topological Markov shifts and a
related class of algebraic integers, {\em
Ergodic Theory Dynam.\ Systems} {\bf 4} (1984), no.\ 2, 283--300.

\bibitem{Maud} C. Mauduit,
Caract\'{e}risation des ensembles normaux substitutifs, {\em Inventiones Math.}
{\bf  95}  (1989),  no.\ 1, 133--147.

\bibitem{Moody} R. V. Moody, Meyer sets and their duals, in {\em The Mathematics of Long-Range Aperiodic
Order (Waterloo, ON, 1995)}, R.~V.~Moody, ed., NATO Adv.\ Sci.\
Inst.\ Ser.\ C Math.\
 Phys.\ Sci., Vol. 489, Kluwer Acad.\ Publ., Dordrecht, 1997, 403--441.

\bibitem{mozes} S. Mozes, Tilings, substitution systems and dynamical systems generated by them, {\em J. d'Analyse Math.} {\bf 53} (1989),
139--186.

\bibitem{Petersen} K. Petersen, Factor maps between tiling dynamical systems,
{\em Forum Math.} {\bf 11} (1999), 503--512.

\bibitem{Prag} B. Praggastis, Numeration systems and Markov partitions from
self similar tilings, {\em Trans.\ Amer.\ Math.\ Soc.} {\bf 351}
(1999), no.\ 8, 3315--3349.

\bibitem{Robi} E. A. Robinson, Symbolic dynamics and tilings of
$\R^d$, Symbolic dynamics and its applications, 81--119, Proc.
Sympos. Appl. Math., 60, Amer. Math. Soc., Providence, RI, 2004.


\bibitem{soltil} B. Solomyak, Dynamics of self-similar tilings,
{\em Ergodic Theory Dynam. Systems} {\bf 17} (1997), 695--738.
Corrections to `Dynamics of self-similar tilings', {\em ibid.}
{\bf 19} (1999), 1685.

\bibitem{solucp} B. Solomyak,
Nonperiodicity implies
unique composition for self-similar translationally finite tilings,
{\em Discrete Comput.\ Geom.} {\bf 20} (1998), no.\ 2, 265--279.


\bibitem{sol-eigen} B. Solomyak, Eigenfunctions for substitution
tiling systems, {\em Advanced Studies in Pure Mathematics} {\bf
43} (2006), 1--22.

\bibitem{Solnotes} B. Solomyak,   Tilings and Dynamics, {\em Lecture Notes}, EMS Summer School on Combinatorics,
Automata and Number Theory, 8-19 May 2006, unpublished manuscript available at\\
\verb+http://www.math.washington.edu/~solomyak/PREPRINTS/notes6.pdf+

\bibitem{Stein} E. M. Stein, {\em Harmonic Analysis}, Princeton University
Press, 1993.

\bibitem{Thur} W. Thurston, AMS lecture notes, 1989.

\bibitem{Vij} T. Vijayaraghavan, On the fractional parts of the powers of a number (II), {\em Proc.\ Cambridge Philos.\ Soc.}
{\bf 37}  (1941), 349--357.

\end{thebibliography}
\end{document}